\documentclass[12 pt]{amsart}
\usepackage{amsmath,amssymb,amsthm}
\usepackage{longtable,}
\usepackage{a4wide}
\usepackage{color}
\usepackage{enumerate}

\usepackage{tikz}
 
\usetikzlibrary{decorations.pathreplacing}

\newtheorem{thm}{Theorem}[section]
\newtheorem*{thm*}{Theorem}
\newtheorem*{conj*}{Conjecture}
\newtheorem{cor}[thm]{Corollary}

\newtheorem{lem}[thm]{Lemma}
\newtheorem{prop}[thm]{Proposition}
\newtheorem{eg}[thm]{Example}
\newtheorem{construction}[thm]{Construction}
\theoremstyle{remark}
\newtheorem{rem}[thm]{Remark}

\theoremstyle{definition}
\newtheorem{defn}[thm]{Definition}
\newcounter{claim}[thm]

\newcommand{\res}{\mathrm{res}}
\newcommand{\wWr}{\,\mathrm{wr}\,}
\newcommand{\Wr}{\,\mathrm{Wr}\,}
\newcommand{\sym}[1]{\mathrm{Sym}(#1)}

\newcommand{\alt}[1]{\mathrm{Alt}(#1)}

\newcommand{\cent}[2]{\text{C}_{#1}(#2)}				

\newcommand{\ba}[1]{\overline{#1}}
\newcommand{\spl}[1]{\mathrm{sp}(#1)}
\newcommand{\diag}{\mathrm{diag}}

\title{A theory of semiprimitive groups \footnote{ T\lowercase{his research forms part of the  \uppercase{A}ustralian \uppercase{R}esearch \uppercase{C}ouncil \uppercase{D}iscovery \uppercase{P}roject   \uppercase{DP}120100446  of the first author and part of the  \uppercase{ARC} \uppercase{D}iscovery \uppercase{E}arly \uppercase{C}areer \uppercase{R}esearch \uppercase{A}ward \uppercase{DE}160100081 of the second.}}}

\author{Michael Giudici}
\author{Luke Morgan}

\address{
School of Mathematics and Statistics (M019)\\
University of Western Australia\\
Crawley, 6009\\
Australia} 

\email{michael.giudici@uwa.edu.au}	\email{luke.morgan@uwa.edu.au}
\subjclass[2010]{Primary 20B05; Secondary 20B07}
\begin{document}

\begin{abstract}
A transitive  permutation group is semiprimitive if each of its normal subgroups is transitive or semiregular. Interest in this class of groups is motivated by two sources: problems arising in universal algebra related to collapsing monoids and the graph-restrictive problem for permutation groups. Here we develop a theory of semiprimitive groups which encompasses their structure, their quotient actions and a method by which all finite semiprimitive groups are constructed. We also extend some results from the theory of primitive groups to semiprimitive groups, and conclude with open problems of a similar nature.
\end{abstract}

\maketitle

\section{Introduction}

Without a doubt, the crowning achievement of 20th Century group theory is the Classification of the Finite Simple Groups (CFSG). This theorem is celebrated not only because of the immense scope of the mathematics that it encompasses, but also because of the light this theorem shines upon many problems in finite group theory. One particular case where the theory has had a successful impact is in applications to problems concerning finite primitive permutation groups. The crux of such applications of the CFSG is the use of the O'Nan-Scott Theorem. In this paper we are concerned with finding an analogous result, an ``O'Nan-Scott type'' theorem, for a wider class of permutation groups, namely, the semiprimitive groups. Our  principal goal upon setting out on this investigation was  to 
find a meaningful subdivision of the class of semiprimitive groups, as in the O'Nan-Scott Theorem, which would allow the CFSG to be
brought to bear upon problems concerning finite semiprimitive groups. 
In this paper we propose a structure theory for semiprimitive groups, which 
in the finite case   is  sufficient  for applications of the CFSG.
In fact, because of the ``wild'' examples  we give in this paper,  we believe our result is  the best possible. Before going into the details of this, we discuss some background. 

A transitive permutation group $G$ on a set $\Omega$ is called \emph{imprimitive} if there exists a $G$-invariant partition of $\Omega$ into more than one part and with each part having size at least two, and \emph{primitive} otherwise.  An equivalent condition to primitivity is that point-stabilisers $G_\omega$ ($\omega \in \Omega$) are maximal in $G$ (i.e.~ that there is no subgroup $H$ of $G$ with $G_\omega < H < G$). The set of orbits of a normal subgroup of a transitive permutation group $G$ forms a system of imprimitivity for $G$, and so all non-trivial normal subgroups of a primitive group are transitive. This leads to a
 natural generalisation where we call a  permutation group \emph{quasiprimitive} if each of its non-trivial normal subgroups is transitive.  Many questions about permutation groups can be reduced to questions about primitive or quasiprimitive groups and they have been the focus of much attention, for example \cite{babai,timseress,cameron,CNT,GMP,HBPS,liebeck,LPSON,Maroti,praegersaxl,praeger2-trans,PLNON,PraegerON,praegerc,praegershalev,pyber,pyberbasesize}.

\emph{Innately transitive} permutation groups were introduced by Bamberg and Praeger \cite{bambergpraeger} and these are the \emph{finite} permutation groups $G$ with a transitive minimal normal subgroup $N$.  Such groups naturally occur as overgroups of quasiprimitive groups. A permutation group $G$ is called \emph{semiregular} if  each point-stabiliser 
 $G_\omega$
is trivial.   It is well known that the centraliser of a transitive group is semiregular \cite[Theorem 4.2A]{DM} and so a normal subgroup of an innately transitive group either contains the transitive minimal normal subgroup $N$, and hence is itself transitive, or intersects $N$ trivially and hence is semiregular.

A permutation group is called \emph{semiprimitive} if every normal subgroup is transitive or semiregular. This notion was introduced by Bereczky and Mar\'oti \cite{bermar}  and was motivated by  an application to collapsing transformation monoids. Their original definition required the group to be non-regular, but here we follow Poto\v{c}nik, Spiga and Verret \cite{psv} and include the regular case. The class  of semiprimitive groups is  much wider than the class of  innately transitive groups and includes all  automorphism groups of graphs that are vertex-transitive and locally quasiprimitive (see Lemma \ref{lem:graph problem}) and all finite Frobenius groups \cite[Lemma 2.1]{bermar}. Note that every normal subgroup of a semiregular group is also semiregular, therefore semiregular groups (whose theory is rather uninteresting) could be considered as the intransitive analogue of semiprimitive groups.

Poto\v{c}nik, Spiga and Verret were interested in semiprimitive groups due to their work on the Weiss Conjecture and its generalisations.  A finite transitive permutation group $L$ is called \emph{graph-restrictive} if there is an absolute constant $c(L)$ such that for any locally $L$ graph-group pair $(\Gamma,G)$,  the order of a vertex stabiliser in $G$ is at most $c(L)$ (see Section \ref{sec:graphs} for more details). The Weiss Conjecture \cite{weissc} asserts that any primitive group is graph-restrictive and has been proved for many classes of primitive groups, for example all 2-transitive groups are graph-restrictive \cite{trofweiss1}. Praeger \cite{praegerc} has conjectured that the class of graph-restrictive groups includes all quasiprimitive groups.
  Poto\v{c}nik, Spiga and Verret \cite{psv} showed that any graph-restrictive group must be semiprimitive and their PSV Conjecture asserts that the converse is also true.

The structure of finite primitive permutation groups is given by the O'Nan-Scott Theorem. Following \cite{PLNON}, this theorem partitions the class of finite primitive groups into eight types and has had a multitude of applications, see for example \cite{bambergetal,timseress,cameron,DM,Maroti,praegerdt,csabahendrik}.   Similar ``O'Nan-Scott  type'' theorems have been developed for the classes of finite quasiprimitive \cite{PraegerON} and innately transitive \cite{bambergpraeger} groups. In the infinite setting, similar structure theorems exist for  infinite primitive permutation groups with a minimal closed normal subgroup that in turn has a minimal closed normal subgroup \cite{macpraeger} and for infinite primitive permutation groups with finite point-stabilisers \cite{simon}. The key feature of these theorems is that each class is divided according to the structure and action of a transitive minimal normal subgroup. Such a subgroup is called a \emph{plinth} by Bamberg and Praeger \cite{bambergpraeger}.  In most cases this enables detailed information about the action of the group and structure of a point-stabiliser. An innately transitive group has at most two plinths, and if it has two plinths then they are isomorphic and regular \cite[Lemma 5.1]{bambergpraeger}.

The aim of this paper is to investigate   semiprimitive groups along the lines of an ``O'Nan-Scott  type'' theorem. We introduce the notion of a \emph{plinth} of a transitive permutation group to be a minimally transitive normal subgroup. Every finite transitive permutation group has a plinth. However, there are infinite primitive groups with no minimal normal subgroups and hence no plinth. (For example, the free group of rank two has a faithful 2-transitive representation \cite{McDonough} but no minimal normal subgroup.) Note that our definition of a plinth is consistent with the definition of a plinth of an innately transitive group. Moreover, any regular normal subgroup is a plinth. 

In \cite{bermar} it was shown that every soluble finite semiprimitive group has a unique regular normal subgroup that contains every semiregular normal subgroup and is contained in every transitive normal subgroup. Such a subgroup was called a \emph{kernel}, for us, it is a plinth. In fact, we arrived at the notion of a plinth in our efforts to extend the work of Bereczky and Mar\'{o}ti.

Whereas the plinth of an innately transitive group is a minimal normal  subgroup and hence the direct product of isomorphic simple groups, the plinth of a semiprimitive group has far fewer restrictions.
In fact, any abstract group is a semiprimitive permutation group acting regularly on itself, and in this action the whole group is a plinth,  so there is no restriction on the structure of a plinth of a semiprimitive group.
Even in the non-regular case, a semiprimitive group can have an arbitrary number of plinths (Example \ref{eg:many plinths}), two plinths need not be isomorphic (Example \ref{eg:nonisoplinth}) and any finite centre-free perfect group can be a non-regular plinth (Example \ref{eg:cfperfect}). However, we are still able to deduce some useful information about the structure of plinths in both the finite and infinite cases.

\begin{thm}\label{thm:plinth}			
Let $G$ be a  semiprimitive group with plinth $K$.
\begin{enumerate}
\item If $K$ is non-regular, then $K$ is perfect and is the unique plinth of $G$.
\item If $L$ is another plinth, then $K/L\cap K \cong L/L\cap K$ is characteristically simple  and every plinth of $G$ is contained in $KL$.
\end{enumerate}
\end{thm}

Theorem \ref{thm:plinth} is proved in Section \ref{sec:structure}, where more information is given about the structure of plinths. In particular, the structure of a  semiprimitive group with two plinths is tightly constrained by Theorem \ref{thm:mult plinths}  and in the finite case any two plinths must have the same set of composition factors. 

Another reason that the structure of semiprimitive groups is less restricted  than that of primitive groups is that there are more ways to build semiprimitive groups. The roughest interpretation of the O'Nan-Scott Theorem says that a primitive permutation group is either a `basic' group, or obtained from a basic group via the product action of   a wreath product. For semiprimitive groups, we have a new kind of product, which we call the \emph{glued product}, which takes two semiprimitive groups   with isomorphic point-stabilisers  and produces a new semiprimitive group by glueing together their point-stabilisers. We make this precise and prove the details in Section~\ref{sec:sptriples}. 

Let $N$ be an intransitive normal subgroup of a semiprimitive group $G$. Then  (as first shown by Bereczky and Mar\'{o}ti in the finite case) $N$ is the kernel of the action of $G$ on the set of $N$-orbits and this action is semiprimitive (Lemma \ref{lem: quotient of sp is sp}). In the finite case, taking $N$ to be maximal subject to containment in a plinth $K$, the group $G/N$ is innately transitive. Thus semiprimitive groups appear to be built out of innately transitive groups. To make this more concrete, we borrow the following notion from representation theory.
If $G$ has at least two plinths, we define $\mathrm{rad}(G)$ to be the intersection of all plinths of $G$. If $G$ has a unique plinth $K$, we define $\mathrm{rad}(G)$ to be the intersection of all proper subgroups of $K$ that are maximal  subject to being normal in $G$.

\begin{thm}
\label{thm:rad g}
Let $G$ be a finite semiprimitive group. Then $G/\mathrm{rad}(G)$ is the glued product of a tightly constrained family of innately transitive groups.
\end{thm}

The above theorem is   a consequence of a more technical result on the structure of semiprimitive groups, namely Theorem~\ref{thm:classification}. The family of groups mentioned in the above theorem is given in explicit detail in Theorem~\ref{thm:classification}.

In Section \ref{sec:sptriples} we introduce the notion of a semiprimitive triple. This is at once a generalisation and a simplification of the   innate triples of Bamberg and Praeger \cite{bambergpraeger}. A semiprimitive triple consists of three groups $K$, $H$ and $L$ with $H$ a group of automorphisms of $K$, and $L$ a normal subgroup of a group $K_0$ of $K$. These three groups satisfy the set of conditions  given in Definition \ref{defn:trips} and can be fed into Construction \ref{con:sptriples} to create a semiprimitive group. All  semiprimitive groups with a plinth can be constructed in this way and we obtain the following theorem.

\begin{thm}
Every  semiprimitive permutation group with a plinth is permutationally isomorphic to a semiprimitive group given by Construction \ref{con:sptriples} and every permutation group given by this construction is semiprimitive.
\end{thm}

We are able to apply the theory that we develop to investigate some properties of semiprimitive groups. In particular,
 we prove the following theorem, which reduces the PSV Conjecture to semiprimitive groups with a unique plinth.
\begin{thm}
\label{thm:intro-grp}
A finite semiprimitive group with at least two plinths is graph-restrictive.
\end{thm}
  We note here that the proof of Theorem~\ref{thm:intro-grp} (given in Section~\ref{sec:graphs}) rests upon a Thompson-Wielandt Theorem, and not the CFSG. As mentioned above, the true power of the O'Nan-Scott Theorems for primitive, quasiprimitive and innately transitive groups lies in the ability to reduce problems on permutation groups in these classes to questions about finite simple groups, and thus to enable a use of the CFSG to solve problems.   Theorem~\ref{thm:rad g} shows that, for a semiprimitive group $G$, the CFSG could be used to answer questions about $G/\mathrm{rad}(G)$. We provide examples in Section~\ref{sec:ex} to show that there exist semiprimitive groups $G$ such that $\mathrm{rad}(G)$  contains   arbitrary finite simple groups as composition factors. Thus bringing the CFSG to bear upon problems concerning semiprimitive groups is feasible, if one can deal with the semiregular normal subgroup $\mathrm{rad}(G)$. 

 A general construction of permutation groups is via the product action of wreath products. Thus we are motivated in  
Section \ref{sec:wreath} 
to investigate wreath products of semiprimitive groups in product action. For the restricted wreath product we are able to give a complete answer to the question of when such a wreath product is semiprimitive. The unrestricted wreath products are more difficult to deal with, and we offer some partial results in this direction. 

 A variety of useful results on primitive groups concern   knowledge of bounds on orders, base sizes and minimal degrees. 
 For more general applications, these types of results  have been extended to quasiprimitive \cite{praegershalev} and innately transitive groups \cite{bamberg}.  
 For some of these results, the extension to semiprimitive groups is rather straightforward -- we give  details  in Section~\ref{sec:properties}. On the other hand, some of these questions run into the wildness of semiprimitive groups, and it is not clear if the expected generalisation of a result from the classes of primitive, quasiprimitive or innately transitive groups holds. In Section~\ref{sec:problems} we discuss some open problems and pose some general questions. These are motivated either by the aforementioned generalisation of results on primitive groups, or by problems that have been raised in our investigations. 

\section*{Acknowledgements}

It is a pleasure to thank Cai Heng Li, Peter Neumann and Cheryl Praeger for various insightful discussions on this topic.

\section{Preliminaries}\label{sec:prelims}

Our notation is mostly standard. We frequently use the \emph{bar notation}, that is, for a group $G$ with normal subgroup $N$ we write $\ba{G}=G/N$ and use the subgroup correspondence theorem to identify subgroups $\ba{H} \leqslant \ba{G}$ with their preimages in $G$.
 For groups $A$ and $B$ and an isomorphism $\mu : A \rightarrow B$, we define the diagonal subgroup of $A \times B$, with respect to $\mu$ as follows:
  $$\mathrm{diag}_\mu(A,B) = \{(a,a\mu) : a\in A\}.$$ 
Further, we refer to any subgroup of the above mentioned form as a diagonal subgroup.

We  recall some basic terminology of permutation groups. 
Let $\Omega$ be a set. A subgroup $G$ of $\mathrm{Sym}(\Omega)$ is referred to as a permutation group (on $\Omega$). The \emph{degree} of $G$ is the cardinality of $\Omega$. The orbit of $\omega \in \Omega$ under $G$ is the set $\omega^G:=\{ \omega^g : g \in G \}$. The group $G$ is \emph{transitive} if $\Omega=\omega^G$ for some $\omega \in \Omega$. For a subset $B \subseteq \Omega$, the \emph{point-wise stabiliser} of $B$ is $G_{(B)}: = \{g \in G \mid \omega^g = \omega$ for all $\omega \in B \}$, while the \emph{set-wise stabiliser} of $B$ is $G_{B} := \{g \in G \mid \omega^g \in B$ for all $\omega \in B \}$. If $B=\{\omega\}$ for some element $\omega$ of $\Omega$ then $G_{(B)}=G_B$ and we simply write $G_\omega$ for this subgroup. We say that $G$ is semiregular if $G_\omega=1$ for each $\omega \in \Omega$ and regular if it is both semiregular and transitive.

Suppose that $G$ is transitive on $\Omega$. A partition $\Delta$ of $\Omega$ is said to be $G$-invariant if for all $\delta \in \Delta$ we have $\delta^g = \{ \omega^g \mid \omega \in \delta \} \in \Delta$. The following two  $G$-invariant partitions are called \emph{trivial}: $\Delta = \{ \{ \omega\} \mid \omega \in \Omega\}$ and $\Delta = \{ \Omega \}$. A partition is therefore called \emph{non-trivial} if it is \emph{not} a trivial partition. The existence of $G$-invariant partitions corresponds to the existence of subgroups $H$ of $G$ such that $G_\omega \leqslant H \leqslant G$, with non-trivial $G$-invariant partitions occurring if subgroups can be found with these inequalities being strict.

In the presence of  a $G$-invariant partition $\Delta$, we can consider two different induced actions of the group $G$. The first is the induced action of $G$ on the set of parts of $\Delta$, as defined above. This gives rise to a homomorphism $G \rightarrow \mathrm{Sym}(\Delta)$. The second is the induced action on a part: if $\delta \in \Delta$ we see that $G_{\delta}$ acts on the set of elements of $\Omega$ in $\delta$. This gives a homomorphism $G_{\delta} \rightarrow \mathrm{Sym}(\delta)$. If $T$ is the permutation group induced on $\Delta$ by $G$ and $M$ the group induced by $G_{\delta}$ on $\delta$, then  $G$ is embedded in the wreath product $M \wr T$ acting  on $\delta \times \Delta$.

The concept of quotient actions is fundamental to understanding semiprimitive groups.

\begin{defn}
Let $G \leqslant \mathrm{Sym}(\Omega)$ and let $\Delta$ be a non-trivial $G$-invariant partition of $\Omega$. The induced action of $G$ on $\Delta$ is called a \emph{quotient action} of $G$ (on $\Delta$). The subgroup of $G$ that fixes each of the parts of $\Delta$ is called the \emph{kernel} of the action. A quotient action is called \emph{faithful} if  the kernel is trivial. If $\Delta$ is the set of orbits of a normal subgroup $N$ of $G$, we refer to the quotient action of $G$ on $\Delta$ as the \emph{quotient action of $G$ via $N$}.
\end{defn}

Note that not all quotient actions are quotient actions via normal subgroups. A primitive permutation group has no quotient actions. A quasiprimitive group may have quotient actions, each of which will be faithful  (the kernel of a quotient action is  necessarily intransitive), and so a quasiprimitive group has no non-trivial quotient actions via normal subgroups. Each innately transitive group that fails to be quasiprimitive automatically has a quotient action via a normal subgroup, for there must exist an intransitive normal subgroup. For each type of innately transitive group, there is in fact a quotient action which will be quasiprimitive -- and the type of this quasiprimitive action is well understood \cite[Table 1]{bambergpraeger}.

We record the following easy facts which we will use without reference.

\begin{lem}
Let $G \leqslant \mathrm{Sym}(\Omega)$ be a transitive permutation group and let $N$ be an intransitive normal subgroup. Then the action of $G$ on the set of $N$-orbits is equivalent to the action of $G$ on the set of cosets of $NG_\omega$, for any $\omega \in \Omega$.
\end{lem}

\begin{lem}
\label{lem:centsemireg}
Let $G \leqslant \mathrm{Sym}(\Omega)$ be a transitive permutation group. Then $C_{\sym{\Omega}}(G)$ is semiregular.
\end{lem}

We now mention some information concerning the types of finite innately transitive groups that pertains to this paper. Recall that the socle of a group $G$ is the product of all the minimal normal subgroups, and is denoted $\mathrm{soc}(G)$. Let $G\leqslant \mathrm{Sym}(\Omega)$ be an innately transitive group and let $\omega \in \Omega$. Then by \cite{bambergpraeger}, $G$ is permutationally isomorphic to a group of exactly one of the following eleven types.

\begin{enumerate}
\item[\textbf{HA}] There is an integer $n$ and a prime $p$ such that $\Omega$ can be identified with $T\cong (\mathrm C_p)^n$ and $G = T \rtimes G_\omega$, where $G_\omega$ acts faithfully and irreducibly on $T$. All groups of this type are primitive and are subgroups of $\mathrm{AGL}(n,p)$.

\item[\textbf{HS}] The groups of this type are primitive, $G$ contains a normal subgroup isomorphic to $T\times T$, where $T$ is a non-abelian simple group. Moreover, $\Omega=T$ and  $G$ is embedded in $\mathrm{Hol}(T)$ with $\mathrm{Inn}(T) \leqslant G_\omega \leqslant \mathrm{Aut}(T)$.

\item[\textbf{HC}] The groups of this type are primitive. There is some integer $k\geqslant 2$ and some non-abelian simple group $T$ such that $\Omega=T^k$ and $T^k \mathrm{Inn}(T^k) \leqslant G \leqslant \mathrm{Hol}(T^k)$ with $G$ transitively permuting the $k$ factors of $T^k$. There are exactly two minimal normal subgroups, each isomorphic to $T^k$  and $\mathrm{soc}(G)_\omega \cong T^k$. 

\item[\textbf{AS}] Here $T \leqslant G \leqslant \mathrm{Aut}(T)$ for some non-abelian simple group $T$. The point-stabiliser  $G_\omega$ is some core-free subgroup of $G$. Groups of this type are quasiprimitive.

\item[\textbf{TW}] Here $\mathrm{soc}(G)$ is regular and isomorphic to $T^k$ for some $k>1$ and some non-abelian simple group $T$. Moreover $C_G(T^k)=1$ and $G$ is transitive on the set of $k$ factors of $T^k$. Groups of this type are quasiprimitive.

\item[\textbf{SD}] The groups of this type are quasiprimitive. Here $\mathrm{soc}(G)=T^k$ with $k>1$, is minimal normal, $G$ is contained in $\mathrm{Aut}(T^k)$ and $\mathrm{soc}(G)_\omega \cong T$  is a full diagonal subgroup.

\item[\textbf{CD}] The groups of this type are also quasiprimitive. There are integers $k,\ell >1$ such that $\mathrm{soc}(G)=T^{k\ell}$ and $\mathrm{soc}(G)_\omega \cong T^k$.

\item[\textbf{ASQ}] There is a non-abelian simple group $T$ such that $T$ is a transitive minimal normal subgroup of $G$,   $C_G(T) \neq 1$ and $C_G(T)$ is not transitive. Groups of this type are not quasiprimitive but have quasiprimitive quotient actions of type AS.

\item[\textbf{PA}] Here $G$ is not necessarily quasiprimitive. There is a $G$-invariant system of imprimitivity $\Sigma$ such that $\Sigma$ can be identified with $\Delta^k$ for some integer $k >1$ and $G^{\Sigma}\leqslant H\wr S_k$ where $H$ is an innately transitive group  on $\Delta$ of type AS or ASQ with non-regular plinth.

\item[\textbf{PQ}] In this case there is a non-abelian simple group $T$ and an integer $k>1$ such that $T^k$ is a regular minimal normal subgroup of $G$ and $C_G(T^k)\neq 1$, and the action induced by $G$ on the set of $C_G(T^k)$-orbits is quasiprimitive of type PA. Groups of this type are not quasiprimitive.

\item[\textbf{DQ}] Groups of this type are not quasiprimitive but have quasiprimitive quotient actions of type SD or CD. There is a transitive regular minimal normal subgroup isomorphic to $T^k$ for some $k>1$, and  $\mathrm{soc}(G)_\omega \cong T^r$ for some $r$. Moreover, $C_G(T^k)\neq 1$ and the action of $G$ on the set of $C_G(T^k)$-orbits is quasiprimitive of type SD or CD.

\end{enumerate}

Additionally we may separate the AS type into the types AS$^{\text{reg}}$ and AS$^{\text{non-reg}}$ corresponding to the property that a minimal normal subgroup is regular or not. The types ASQ$^{\text{reg}}$ and ASQ$^{\text{non-reg}}$ similarly partition the type of ASQ groups.

\begin{lem}
\label{lem:criterion for prim hs hc}
Suppose that $G \leqslant \mathrm{Sym}(\Omega)$ has two distinct minimal normal subgroups $L$ and $R$ that are transitive. Then $G$ is primitive and there is a non-abelian characteristically simple group $X$ such that $X\cong L\cong R\cong (LR)_\alpha$ for $\alpha\in\Omega$. Moreover, $L$ and $R$ are the only minimal normal subgroups of $G$, and if $G$ is finite then $G$ is of type HS or HC.
\end{lem}
\begin{proof}
Note first that both $L$ and $R$ are regular by the lemma above. Then since $L\neq R$, both $L$ and $R$ are non-abelian and we have $C_{\sym{\Omega}}(L) =R$ and $C_{\sym{\Omega}}(R) = L$. In particular, $L$ and $R$ are the only minimal normal subgroups of $G$. By \cite[Lemma 4.2A(ii)]{DM}  there is a
group  $X$ such that $L$ is the left regular representation of $X$ and $R$ is the right regular representation of $X$. Since $L$ is a minimal normal subgroup of $G$, $X$ is characteristically simple. Let us write $\sigma : X \rightarrow L$ and $\rho : X \rightarrow R$ so that for $g\in X$ we write $\sigma : g \mapsto \sigma_g$ and $\rho : g \mapsto \rho_g$.  Further, we may identify $\Omega$ with $X$ such that, for $\sigma_\ell \in L$ and $\rho_r \in R$ and $g\in \Omega$, we have 
$$g\sigma_\ell =\ell^{-1} g \quad \text{and} \quad g\rho_r = gr.$$
Suppose that $B$ is a   block of imprimitivity of $\Omega$. Since $G$ is transitive, we may assume that $1\in B$. Suppose that $x\in B$. Then $x^{-1} = 1\sigma_x \in B\sigma_{x}$. Also $x \sigma_x= x^{-1}x =1 \in B\sigma_{x}$ so that $B=B \sigma_{x}$.  Similarly, $B=B\rho_{x}$. Hence if $x, y\in B$ then $B=(B\sigma_{x})\rho_{y}=B(\sigma_x \rho_{y })$ so that  $x^{-1}y=x^{-1}1y  \in B$. Hence $B$ is a subgroup of $X$. Let $g\in X$, then $1^g = 1$, so $1 \in B\sigma_{g}\rho_g$ so that $g^{-1}Bg = B$. Hence $B$ is a normal subgroup of $X$. Finally, since $L$ is a minimal normal subgroup of $G$, and $B$ is normalised by $G_1$, $B$ is normalised by $G=LG_1$, so  either $B=L=\Omega$ or $B=1$. Hence $G$ is primitive.

Let $K=LR\cong L\times R$ and define $\pi_1$ and $\pi_2$ to be the projections of $K$ onto $L$ and $R$ respectively. Since $L\cap R=1$ we have that $\ker(\pi_1)=R$ and $\ker(\pi_2)=L$. Let $\alpha\in\Omega$ and observe that the transitivity of $L$ and $R$ imply $K=LK_\alpha=RK_\alpha$. Thus $L=\pi_1(K)=\pi_1(K_\alpha)$ and $R=\pi_2(K)=\pi_2(K_\alpha)$.  Since $L$ and $R$ are regular we have that $K_\alpha\cap L=1$ and $K_\alpha\cap R=1$. Thus $K_\alpha\cong\pi_1(K_\alpha)\cong \pi_2(K_\alpha)$ and the result follows.
\end{proof}

\begin{lem}
\label{lem:prim of type hs hc}
Let $G \leqslant \mathrm{Sym}(\Omega)$ be a finite primitive group of type HS or HC and let $N$ be a minimal normal subgroup of $G$. Then there is a non-abelian simple group $S$ and an integer $\ell$ such that, for any $\omega \in \Omega$, we have $N \cong S^{\ell} \cong \mathrm{soc}(G_\omega)$.
\end{lem}
\begin{proof}
This follows from the description of the primitive groups of types HS and HC given above.
\end{proof}

 In the lemma below, and later in this paper,   we write $c_x$  ($x\in K$) for the automorphism of $K$ induced by conjugation by $x$.

\begin{lem}
\label{lem:cent lemma}
Suppose that $G \leqslant \mathrm{Sym}(\Omega)$ and $K$ is a transitive normal subgroup of $G$. Let $\omega \in \Omega$ and let $\sigma = \omega^{C_G(K)}$. Then for each $k\in K_\sigma$ there is a unique $r_k \in C_G(K)$ such that $kr_k \in G_\omega$. Further:
\begin{enumerate}[(a)]
\item the map $\phi : K_\sigma \rightarrow C_G(K)$ defined by $\phi: k \mapsto r_k$ is a surjective group homomorphism with kernel $K_\omega$;
\item with $\mu : G_\omega \rightarrow \mathrm{Aut}(K)$ the natural map,
 $$K_\sigma = \{k \in K \mid\text{ there is }h\in G_\omega\text{ such that }c_k=h\mu \}.$$
\end{enumerate}
\end{lem}
\begin{proof}
Since $K$ is transitive on $\Omega$, we have that $C_G(K)$ is semiregular. Hence, for each $k\in K_\sigma$ there is a unique element $r_k \in C_G(K)$ such that $\omega k = \omega r_k^{-1} $, that is, such  that 
$$k r_k \in  (K_\sigma C_G(K))_\omega.$$
For each $k\in K_\sigma$ define $k\phi = r_k$. We claim that $\phi$ is a homomorphism.
Let $k, h\in K_\sigma$ and note that $k(k\phi) h h\phi = kh (k\phi h\phi) \in (K_\sigma C_G(K))_\omega$. Now $kh [(kh)\phi] \in (K_\sigma C_G(K))_\omega$ hence $(kh)^{-1} ((kh)\phi)^{-1} = h^{-1}k^{-1} ((kh)\phi)^{-1} \in (K_\sigma C_G(K))_\omega$. Thus
$$[kh (k\phi h\phi) ] [ h^{-1}k^{-1} ((kh)\phi)^{-1} ] = (k\phi h\phi) ((kh)\phi)^{-1} \in (K_\sigma C_G(K))_\omega \cap C_G(K) = C_G(K)_\omega.$$
Since $C_G(K)_\omega = 1$, we obtain $ k \phi h \phi  = (kh)\phi$ as required.
Since $k\phi = 1$ implies $\omega k=\omega (k\phi)=\omega$ we have that $\ker \phi = K_\omega$. Clearly $\phi$ is surjective  and (a) is established.

For part (b), for each $k\in K_\sigma$, we have that $k(k\phi) \in G_\omega$. Hence $[k (k\phi)]\mu = k\mu =h \mu$ for some $h\in G_\omega$. Conversely, if $k\in K$ and there is $t \in G_\omega$ such that $c_k = t\mu$, then $kt^{-1} \in C_G(K)$ and so $k\in C_G(K)G_\omega = G_\sigma$. Thus $k\in K \cap G_\sigma = K_\sigma$.
%
%
\end{proof}


We record the following well-known fact  and provide a proof for completeness. In the proof, we use the following notion: A \emph{normal section} of a group $G$ is a quotient $K/N$ where $K$ and $N$ are normal subgroups of $G$. Naturally $G$ acts by conjugation on $K/N$ and we call the normal section $G$-\emph{simple} if there are no proper non-trivial subgroups of $K/N$ invariant under this action of $G$, equivalently, if $K/N$ is a minimal normal subgroup of $G/N$.

\begin{prop}
\label{prop:duality}
Let $G$ be a finite group with normal subgroups $N$ and $K$ with $N \leqslant K$. Let $\mathcal S =\{M_1,\dots,M_r\} $ be a set of  normal subgroups of $G$ that are maximal with respect to $N \leqslant M_i < K$ and set $S = \bigcap _{M \in \mathcal S} M$. Then the following hold:
\begin{enumerate}[(i)]
\item $K/S= L_1 \times \dots \times L_s$, where each $L_i$ is a minimal normal subgroup of $G/S$ and $s\leqslant r$;
\item $L_i \cong K/M$ for some $M\in \mathcal S$;
\item if $K/S$ is perfect, then $s=r$.
\end{enumerate}
\end{prop}
\begin{proof}
We apply induction on $r$. If $r=1$ then there is nothing to prove. Suppose that $r>1$ and let $\ba{G}=G/S$. Let $\ba{L}$ be a minimal normal subgroup of $\ba{G}$  contained in $\ba{K}$. Since $\ba{L}$ is non-trivial, there is some $M \in \mathcal S$ such that $\ba{L} \nleqslant \ba{M}$, so  $\ba{K} = \ba{M} \times \ba{L}$. Let $\mathcal T = \{ \ba{M} \cap \ba{J} \mid J \in \mathcal S - \{M \} \}$. Then $\mathcal T$ is a set of at most $r-1$ normal subgroups of $\ba{G}$, and it is easy to check that $\ba{M \cap J} = \ba{M} \cap \ba{J}$ is maximal with respect to $1 \leqslant \ba{M} \cap \ba{J} < \ba{M}$. Hence by induction, $\ba{M} = \ba{L_1} \times \dots \times \ba{L_s}$ for some minimal normal subgroups $\ba{L_1}$, \dots, $\ba{L_s}$ of $\ba{G}$ with $s \leqslant r-1$. Hence $\ba{K} = \ba{M} \times \ba{L} = \ba{L} \times \ba{L_1} \times \dots \times \ba{L_s}$ is the product of at most $r$ minimal normal subgroups. For the final part, suppose that $\ba{K}$ is perfect. Hence  $\ba{M}$ is perfect, and so it suffices to prove that $|\mathcal T|=r-1$. Assume that $\ba{M} \cap \ba{J} = \ba{M} \cap \ba{H}$ for some $H,J \in \mathcal S$ with $H$, $M$ and $J$ all distinct. Thus $H\cap M = J \cap M = H \cap J \cap M.$
Now $K/M=MH/M \cong H / H \cap M$ is $G$-simple, and 
 $$H \cap M = H\cap M \cap J \leqslant H \cap J \leqslant H,$$
so either $H\cap J = H$ which implies $H=J$, a contradiction, or $H \cap J = H \cap M \cap J$. In the latter case, we have $[M,K]=[M,HJ] = [M,J][M,H] \leqslant M \cap J \cap H$ so that $M/(M \cap J \cap H)$ is in the centre of $K/(M \cap J \cap H)$. In particular, $\overline{M}$ is not perfect, a final contradiction.
\end{proof}

\section{Structure theory for semiprimitive groups}
\label{sec:structure}

The first result shows that the class of semiprimitive groups is closed under quotient actions. 
The first half of the following is due to \cite[Lemma 2.4]{bermar}.

\begin{lem}
\label{lem: quotient of sp is sp}
Let $G \leqslant \sym{\Omega}$ be a semiprimitive group and let $N$ be an intransitive normal subgroup. Then $G/N$ acts faithfully and semiprimitively on the set of $N$-orbits. Moreover, for any $\omega \in \Omega$ and  $\delta$ the $N$-orbit containing $\omega$, we have $G_\omega \cong (G/N)_\delta$.
\end{lem}
\begin{proof}
Let $M$ be the kernel of the quotient action. Then $M$ is a non-trivial normal intransitive subgroup of $G$. Thus $M$ is semiregular and has the same orbits as $N$, therefore $M=N$.
  Let $\ba{G}=G/N$ and suppose that $\ba{R}$ is a non-semiregular normal subgroup of $\ba{G}$. Then $\ba{R} \cap \ba{G_\omega}$ is non-trivial for some $\omega \in \Omega$. Thus  $R \cap NG_\omega=N(R \cap G_\omega)$ is not contained in $N$. Hence $R\cap G_\omega \neq 1$, and so $R$ is a non-semiregular normal subgroup of $G$. Thus $R$ is transitive on $\Omega$. Hence $R$ is transitive on the set of $N$-orbits, and so $\ba{R}$ is transitive.

Let $\omega \in G$ and let  $\delta$ be the $N$-orbit containing $\omega$. Then $(G/N)_\delta=(G_\omega N)/N$. Since $N$ is semiregular, using an isomorphism theorem, we have $(G_\omega N)/N \cong G_\omega / (G_\omega \cap N) \cong G_\omega$.
\end{proof}

The following fact is sometimes useful.

\begin{lem}
\label{lem:[g,r] trans}
Let $G \leqslant \sym{\Omega}$ be semiprimitive and let $R$ be a non-trivial normal subgroup of $G_\omega$ for $\omega \in \Omega$. Then $[G,R]$ is a transitive normal  subgroup of $G$.
\end{lem}
\begin{proof}
The subgroup $[G,R]R$ is normalised by $G$ and is non-semiregular since $R\neq 1$. Hence $[G,R]R$ is transitive, and so $G=[G,R]R G_\omega = [G,R] G_\omega$. Thus $[G,R]$ is transitive.
\end{proof}

The following gives sufficient conditions for a group to be semiprimitive.

\begin{lem}
\label{lem:g is semi prim}
Let $G$ be a group with a normal subgroup $K$ and a core-free subgroup $H$ such that $G=KH$ and
 $[K,N] = K$ for each non-trivial normal subgroup $N$ of $H$.
Then $G$ is semiprimitive on the set of cosets of $H$ in $G$.  
\end{lem}
\begin{proof}
Since $H$ is core-free in $G$, we view $G$ as a permutation group on the set of (right) cosets of $H$ in $G$. Let $T$ be a normal subgroup of $G$. 
If $K \leqslant T$ then $TK \geqslant KH = G$, so $T$ is transitive. Suppose now that  $T \cap K < K$. 
Suppose that $1 \neq T \cap H$. Then $T\cap H$ is a non-trivial normal subgroup of $H$ and so
$$K = [T \cap H, K] \leqslant [T, K] \leqslant T \cap K < K,$$ a contradiction. Hence $T \cap H=1$, that is, $T$ is semiregular.
\end{proof}

The converse to the lemma is false for arbitrary normal subgroups. For example, for any integer $n \geqslant 3$, the action of $G=\mathrm{Sym}(n)$ on the set of cosets of a subgroup $H$ generated by a transposition is semiprimitive (in fact, quasiprimitive). However if  one takes $K=G$ then we have that $[K,H] = [K,K]$ has index two in $K$. The search for a converse however leads us to the central concept in our theory of semiprimitive groups and requires analysis of transitive normal subgroups. 

We fix now a (possibly infinite) set $\Omega$ and a transitive permutation group $G \leqslant \mathrm{Sym}(\Omega)$. Let $\omega \in \Omega$ and set $H=G_\omega$.

\begin{defn}
\label{defn:plinths}
A \emph{plinth} of $G$ is a minimally transitive normal subgroup. The product of all the plinths of $G$ is the \emph{superplinth} of $G$ and is denoted $\mathrm{sp}(G)$.

If $\mathrm{sp}(G)$ is a plinth, we define $\mathrm{rad}(G)$ to be the intersection of all proper subgroups of $\spl{G}$ that are maximal with respect to being normal in $G$.

If $\spl{G}$ is not a plinth, we define $\mathrm{rad}(G)$ to be the intersection of all the plinths of $G$.
\end{defn}

The above definition is motivated by the theory of innately transitive groups, where a transitive minimal normal subgroup is called a plinth \cite[pg.~71]{bambergpraeger}. Note that our more general definition of a plinth agrees with that of loc.~cit.~in the case of innately transitive groups.  It is immediate that every \emph{finite} transitive group has a plinth, although this may be the whole group (as in the case of regular permutation groups).   We gave an example of an infinite 2-transitive (and hence semiprimitive) group with no plinth in the introduction.

As mentioned above, the search for a converse to Lemma~\ref{lem:g is semi prim} leads us to consider plinths. The next result says that semiprimitive groups are characterised by the action of a point-stabiliser on a plinth.
Recall that the kernel of the action of $H$ on a $H$-invariant quotient $K/Y$ of $K$ is 
$$ C_H(K/Y) = \{ h \in H \mid [K,h] \leqslant Y \}.$$
Note that $H$ acts faithfully on every quotient $K/Y$ of $K$ if and only if $K = [K,N]$ for all normal subgroups $N$ of $H$.

\begin{lem}
\label{lem:sp iff h faithful on k/y}
Suppose that $K$ is a plinth of $G$. Then $G$ is semiprimitive if and only if $H$ acts faithfully on $K/Y$ for each normal subgroup $Y$ of $G$ properly contained in $K$. 
\end{lem}
\begin{proof}
Since $K$ is normal in $G$, one of the implications follows from Lemma~\ref{lem:g is semi prim}. Suppose that $G$ is semiprimitive. Let $Y$ be an arbitrary proper  subgroup of $K$ that is normal in $ G$.  Since $K$ is a plinth, $Y$ is intransitive and therefore semiregular. Now let $B=\mathrm C_H(K/Y)$ and observe that $B$ is a normal subgroup of $H$. We claim that $YB$ is a normal subgroup of $G$. Indeed, note that $YB$ is normalised by $H$, and
$$[YB,K] \leqslant [Y,K][B,K] \leqslant Y$$
so that $YB$ is normalised by $KH=G$. Hence $YB$ is transitive or semiregular. In the first case, we have $G=YBH=YH$, and so $Y$ is transitive, a contradiction to $K$ being a plinth. Hence $YB$ must be semiregular, which implies $B=1$ and so $H$ acts faithfully on $K/Y$.
%
\end{proof}

For the rest of this section we adopt:
\begin{center} \textbf{Hypothesis:} $G\leqslant \mathrm{Sym}(\Omega)$ is semiprimitive.\end{center}
 A fruitful approach to  proving statements concerning semiprimitive groups is to consider quotient actions. Usually we are concerned with properties of plinths, and here some caution is required, since the image of a plinth of $G$ may not be a plinth in a quotient action of $G$. The following result summarises the properties that we will draw upon.

\begin{lem}
\label{lem:plinths in quotient actions}
Let $K$ be a plinth of $G$ and let $M$ be an intransitive normal subgroup of $G$. Let $\Delta$ be the set of $M$-orbits and let $\ba{G}=G/M$. The following hold:
\begin{enumerate}[(a)]
\item if $M \leqslant K$, then $\ba{K}$ is a plinth of $\ba{G}$;
\item if $K$ is non-regular, then $\ba{K}$ is non-regular;
\item if $K$ is regular, then $\ba{K}$ is regular if and only if $M \leqslant K$;
\item if every transitive normal subgroup of $G$ contains $K$, then $K$ is the unique plinth of $G$ and $\ba{K}$ is the unique plinth of $\ba{G}$.
\end{enumerate}
\end{lem}
\begin{proof}
Let $\delta \in \Delta$ and let $\alpha \in \delta$. For (a) suppose that $M \leqslant K$ and let $\ba{Y}$ be a transitive normal subgroup of  $ \ba{K}$. Then the preimage $Y$ of $\ba{Y}$ is a normal subgroup of $G$ and is transitive on the set of $M$-orbits. Since $Y$ contains $M$ we have that $Y$ is transitive. Since $Y\leqslant K$ and $K$ is a plinth, we have $Y=K$ and hence $\ba{Y} = \ba{K}$ so that $\ba{K}$ is a plinth of $\ba{G}$. 

For (b), if $1\neq k \in K_\alpha$, then since $M$ is semiregular, $1\neq \ba{k}$ and $\ba{k} \in \ba{MK_\alpha} = \ba{K}_\delta$. Hence $\ba{K}$ is non-regular.

For (c), if $K$ is regular, then $MK$ is non-regular (on $\Omega$) if and only if $M \nleqslant K$. Thus $\ba{K}$ is regular if and only if $M \leqslant K$.

For (d),  note that $K=\mathrm{sp}(G)$, that is,  $K$ is the unique plinth of $G$.  Let $\ba{R} \leqslant \ba{G}$ be an arbitrary plinth of $\ba{G}$. Then $R$ is a transitive normal subgroup of $G$ and thus,  $R$ contains $K$. Hence $\ba{K} \leqslant \ba{R}$. Since $\ba{R}$ is a plinth of $\ba{G}$, and since $\ba{K}$ is a transitive normal subgroup, we have $\ba{R}=\ba{K}$.  In particular, $\ba{K}$ is a plinth of $\ba{G}$ and thus $\ba{K}=\mathrm{sp}(\ba{G})$ as required.
\end{proof}


We next find a description of the superplinth.

\begin{lem}
\label{superplinth}
Suppose that $N$ is a non-regular normal subgroup of $G$. Then $N$ contains every plinth of $G$. In particular,  if $G$ has at least one plinth, then  either there is a unique plinth of $G$ or, for any two distinct plinths $K$ and $L$, we have $KL=\mathrm{sp}(G)$.
\end{lem}
\begin{proof}
Let $K$ be a plinth and let $H$ be a point-stabiliser. Now $N \cap H$ is non-trivial and normal in $H$, so Lemma~\ref{lem:sp iff h faithful on k/y} shows that $K = [K, N \cap H]$. Hence the normality of $N$ in $G$ yields
$$K = [K, N \cap H] \leqslant [K,N] \leqslant N.$$

Now suppose that $G$ has a plinth, $K$ say, and let $S=\mathrm{sp}(G)$. Assume that $L$ is a plinth of $G$  distinct from $K$. Then $KL \leqslant \mathrm{sp}(G)$. If $KL$ is regular, then we would have $K=KL=L$, a contradiction. Hence $KL$ is non-regular, and so the above paragraph shows that every plinth is contained in $KL$. Hence $\mathrm{sp}(G) \leqslant KL$ and we are done.
\end{proof}

 By Lemma \ref{lem:criterion for prim hs hc},  innately transitive groups have at most two plinths. Whilst the above result says that the superplinth is the product of at most two plinths, the following example shows there is no  bound on the number of plinths in semiprimitive groups.

\begin{eg}
\label{eg:many plinths}
Let  $T$ be a  non-abelian simple group, let $I$ be a set and for each $i\in I$ let $T_i$ be a copy of $T$. Let 
$$G = \prod _{i\in I} T_i$$
 and let $H$ be a diagonal subgroup of $G$ isomorphic to $T$. Then $G$ is a semiprimitive group on the set of cosets of $H$ since each proper normal subgroup of $G$ is semiregular. Let  $i\in I$, then the subgroup 
$$K_i=\prod_{j\neq i} T_j$$ is a regular plinth of $G$.   Note that $\mathrm{sp}(G)=G$. 
\end{eg}

When $G$ has a non-regular plinth the situation is quite different.

\begin{lem}
\label{solplinthsreg}
Suppose that $G$ has a non-regular plinth  $K$. Then the following hold:
\begin{enumerate}[(i)]
\item  $K=\spl{G}$;
\item $K$ is perfect;
\item   $K$ is contained in every transitive normal subgroup.
 \end{enumerate}
\end{lem}
\begin{proof}
By Lemma~\ref{superplinth} we have that $K$ contains all plinths, hence $K$ must be the unique plinth so $K=\mathrm{sp}(G)$. Since $K$ is non-regular,  $K\cap H \neq 1$ and so Lemma~\ref{lem:sp iff h faithful on k/y} shows that $K = [K, K \cap H] \leqslant [K, K] \leqslant  K$, hence $K$ is perfect.
Since a non-regular transitive normal subgroup contain every plinth by Lemma~\ref{superplinth}, part (iii) follows immediately from (i).  %
\end{proof}


We now consider the case that $G$ has at least two regular plinths, the previous result shows that all plinths are therefore regular. Here we have a satisfactory reduction to primitive groups.

\begin{thm}
\label{thm:mult plinths}
Suppose that $G$ has at least two plinths. Then there is a characteristically simple group $X$ such that 
 for any two plinths $L$ and $R$  of $G$, $G/ L \cap R$ is a primitive group  with $L/L\cap R\cong X \cong R/L\cap R$. 
 In particular, if $G$ is finite, then there is a uniquely determined non-abelian simple group $T$ and an integer $\ell$ such that $G/L\cap R$ is primitive of type HS or HC with socle $T^\ell \times T^\ell$.
\end{thm}
\begin{proof}
Let $L$ and $R$ be distinct plinths of $G$. Then both $L$ and $R$ are regular by Lemma~\ref{solplinthsreg}(i). Let $\ba{G}=G/L \cap R$. By Lemma~\ref{superplinth} we have that $\spl{G}=LR$. Suppose that $K$ is a plinth of $G$ and $U$ is a normal subgroup of $G$ such that $K < U \leqslant \spl{G}$. Then $U$ is non-regular, and so Lemma~\ref{superplinth} shows that $\spl{G} \leqslant U$, hence $U=\spl{G}$. Thus $\spl{G}/L$ and $\spl{G}/R$ are minimal normal subgroups of $G/L$ and $G/R$ respectively. Since $\spl{G}/L \cong L / L\cap R$ and $\spl{G}/R \cong R /L\cap R$ we have that $\ba{R}$ and $\ba{L}$  are minimal normal subgroups of $\ba{G}$. Since $ L \cap R$ is intransitive, Lemma~\ref{lem: quotient of sp is sp} shows that  $\ba{G}$ is semiprimitive. Now $\ba{G}$ has two distinct minimal normal subgroups which are transitive. Thus Lemma~\ref{lem:criterion for prim hs hc} shows that $\ba{G}$ is primitive with $\ba{L}\cong\ba{R}$. Set $X:=\ba{L}$, by Lemma~\ref{lem:criterion for prim hs hc} $X$ is characteristically simple. Note that  if $G$ is finite, then $\ba{G}$ is of type HS or HC and  $X \cong T^\ell$ for some integer $\ell$ and some finite non-abelian simple group $T$.

We now show that $X$ does not depend upon the choice of plinths. Let $K$ and $Y$ be distinct plinths of $G$ and set $\widetilde{G}= G/ K \cap Y $. Lemma~\ref{lem: quotient of sp is sp} shows that, for any $\alpha \in \Omega$, $(KY)_\alpha \cong \widetilde{KY}_{\alpha^{K \cap Y}}$. Now $KY=\spl{G}=LR$ so we have $(KY)_\alpha = (LR)_\alpha \cong (\overline{LR})_{\alpha^{L\cap R}}$. Lemma~\ref{lem:criterion for prim hs hc} shows that $X \cong (\overline{LR})_{\alpha^{L \cap R}}$, hence $\widetilde{KY}_{\alpha^{K \cap Y}} \cong X$ and thus Lemma~\ref{lem:criterion for prim hs hc} shows that  $\widetilde{K} \cong \widetilde{Y} \cong X$.
\end{proof}

\begin{cor}
\label{plinthfactors}
All plinths of a  finite semiprimitive group $G$ have the same multiset of composition factors.
\end{cor}
\begin{proof}
Obviously there is nothing to prove if $G$ has a unique plinth,  whilst Theorem~\ref{thm:mult plinths} shows that for any two distinct plinths $L$ and $R$ we have  $L/L\cap  R\cong R/L\cap R$ and hence $L$ and $R$ have the same set of composition factors.
\end{proof}

Having a non-perfect plinth gives us control over centralisers and thus over certain ``useful'' subgroups of $G$. In the next lemma, the \emph{layer} of a finite group $G$, $\mathrm E(G)$, is the product of all subnormal quasisimple subgroups of $G$.

\begin{lem}
\label{lem:non-perfect plinth subgroup}
Suppose that $K$ is a plinth of $G$ with $[K,K] < K$.  Then $K$ is regular and $\cent{G}{K} \leqslant K$.  In particular, if $G$ is finite, then $\mathrm E({G}) \leqslant K$.
\end{lem}
\begin{proof}
That $K$ is regular is  a consequence of Lemma~\ref{solplinthsreg}(ii).  Assume for a contradiction that $\cent{G}{K}K >K$. Then there is a normal subgroup $1\neq S$ of $H$ such that $SK=K\cent{G}{K}$, whence 
$$[S,K] \leqslant [SK,K] = [K\cent{G}{K},K]= [K,K] < K.$$
 Now $S$ acts trivially on $K/[S,K]$, so $K/[S,K]$ is a  $H$-invariant quotient of $K$ on which $H$ does not act faithfully, a contradiction to Lemma~\ref{lem:sp iff h faithful on k/y}. Thus $K$ indeed contains its centraliser. The final part of the lemma follows since for every finite group $X$, $\mathrm E({X}) \leqslant U$ for every normal subgroup $U$ of $X$ such that $\cent{X}{U} \leqslant U$, see \cite[9.A.6]{Isaacs}.
\end{proof}

\begin{thm}
Suppose that $K$ is a soluble plinth. Then the following hold:
\begin{enumerate}[(i)]
\item  $K$ is regular and $K=\mathrm{sp}(G)$;
\item every semiregular normal subgroup of $G$ is contained in $K$;
\item every transitive normal subgroup of $G$ contains $K$.
\end{enumerate}
\end{thm}
\begin{proof}
Suppose that $G$ is a semiprimitive group with soluble plinth $K$. Suppose that   $L$ is also a plinth of $G$. If $L\neq K$, then  Theorem \ref{thm:mult plinths} shows there is a   non-abelian characteristically simple group $X$ such that $K/K\cap L \cong X$,  a contradiction to $K$ being soluble. Hence $K=L$ and so $K=\mathrm{sp}(G)$. Lemma~\ref{solplinthsreg} shows that $K$ is regular and so (i) holds.

For (ii), suppose  that $R$ is a normal semiregular subgroup of $G$. Assume that $R \nleqslant K$. Then $[R,K] \leqslant R \cap K < K$. Since $K$ is soluble, $K/K \cap R$ is soluble, and so $[K,K](K \cap R)$ is a proper subgroup of $K$. Now $KR = KS$ for some normal subgroup $S$ of $H$. Then 
$$[S,K] \leqslant [SK,K] =       [KR,K] = [K,R] [K,K] < K,$$
a contradiction to Lemma~\ref{lem:sp iff h faithful on k/y}. Hence $R \leqslant K$ as required. 

Let $N$ be a transitive normal subgroup of $G$. If $N$ is non-regular, then $N$ contains $K$ by Lemma~\ref{superplinth}. Otherwise, $N$ is regular, and is therefore a plinth, and therefore $N=K$ by part  (i).
\end{proof}

The following example (suggested to us by Cai Heng Li) shows that solubility of the plinth is necessary in the above result.

\begin{eg}
\label{eg:nonisoplinth}
Let $T$ be a non-abelian finite simple group and let $V$ be a faithful irreducible module for $T$ over some finite field. Let
$$ X = \{(v,a,b) \mid v \in V, a,b \in T \}\cong (V\times T)\rtimes T$$
with multiplication defined by $(vab)(wcd)=vw^{b^{-1}}ac^{b^{-1}}bd$.
Identify $V$ with the subgroup $\{(v,1,1) \mid v\in V \}$ of $X$ and let $R=\{(1,a,1) \mid a \in T\} \cong T$. Set
 $$K = VR,\ L  = \{ (v,a,a^{-1}) \mid v\in V,a\in T\},\ H=\{(1,1,a) \mid a\in T\}.$$
Then $X$ acts faithfully on the set of cosets of $H$. Both $L$ and $K$ are regular subgroups in this action. The non-trivial proper normal subgroups of $X$ are simply $V$, $R$, $K$ and $L$, and since $V$ and $R$ are semiregular, $X$ is semiprimitive. Note that $K$ is non-perfect whilst $L$ is perfect. Both $L$ and $K$ are insoluble.

It is worth noting that $X$ has two quotient actions via normal subgroups, $X/V$ is a primitive group of type HS and $X/R$ is a primitive group of type HA.
\end{eg}

\section{Construction of semiprimitive groups via triples}
\label{sec:sptriples}

One of the main ideas in \cite{bambergpraeger} is an encoding of the building blocks of an innately transitive group into a data set  called an \emph{innate triple}.

\begin{defn}
\label{defn:intrips}
A triple $(K,H,\phi)$ satisfying the following three conditions is called an \emph{innate triple}:
\begin{enumerate}
\item $K \cong T^k$ with $T$ a finite simple (possibly abelian) group.
\item $\phi$ is an epimorphism with domain a subgroup $K_0$ of $K$ such that $\ker(\phi)$ is core-free in $K$ and if $K$ is abelian then $K_0 = K$;
\item $H$ is a subgroup of $\mathrm{Aut}(K)$ such that $K$ is $H$-simple, $\ker(\phi)$ is $H$-invariant and $H \cap \mathrm{Inn}(K) = \mathrm{Inn}_{K_0}(K):=\{c_x \mid x \in K_0\}$.
\end{enumerate}
\end{defn}

A rigorous method was developed in \cite{bambergpraeger} to show that all innately transitive groups arise from an innate triple via a construction. Here we show that there is an  appropriate  generalisation of both the notion of innate triples and  the construction to the  case of semiprimitive groups with a plinth, although our treatment is somewhat simplified compared to that of \cite{bambergpraeger}.

\begin{defn}
\label{defn:trips}
A triple $(K,H,L)$ satisfying the following conditions is called a \emph{semiprimitive triple}:
\begin{enumerate}
\item $K$ is a group and  $H$ is a group of automorphisms of $K$ such that $H$ acts faithfully on each non-trivial $H$-invariant quotient of $K$;
\item  $L$ is a normal subgroup of $K_0 :=\{ k\in K \mid   c_k \in H   \},$
$L$ is core-free in $K$, $L$ is normalised by $H$ and, if $L\neq 1$, then $K=[K,K]$;
\item $K \neq LR$ for any proper normal subgroup $R$ of $K$ that is $H$-invariant.
\end{enumerate}
\end{defn}

 Note that $H$ normalises $K_0$ and that $K_0$ contains $Z(K)$.

Naturally, semiprimitive groups give rise to semiprimitive triples.

\begin{lem}
\label{lem:triple from sp gp is sp trip}
Let $G \leqslant \mathrm{Sym}(\Omega)$ be a semiprimitive group with plinth $K$ and let $\omega \in \Omega$. Let $\mu : G \rightarrow \mathrm{Aut}(K)$ be the natural map induced by the conjugation action of $G$ on $K$. Then $G_\omega \cong  G_\omega \mu$ and $(K, G_\omega \mu,K_\omega)$ is a semiprimitive triple.
\end{lem}
\begin{proof}
First note that $\ker(\mu) = C_G(K)$ and since $K$ is transitive, $C_G(K)$ is semiregular, hence $G_\omega \cap C_G(K) = 1$ so that $G_\omega$ is indeed isomorphic to $G_\omega\mu$. Since $K$ is a plinth, Lemma~\ref{lem:sp iff h faithful on k/y} shows that $G_\omega\mu$ acts faithfully on each non-trivial $G_\omega$-invariant quotient of $K$. Hence the triple $(K,G_\omega\mu,K_\omega )$ satisfies Definition~\ref{defn:trips}(1).

Note that $K_\omega$ is core-free in $K$ since $K$ acts faithfully on $\Omega$. Also since $K_\omega = K \cap G_\omega$ and $K$ is normal in $G$, we have that $K_\omega$ is $G_\omega \mu$-invariant. Let $\sigma = \omega^{C_G(K)}$. Lemma~\ref{lem:cent lemma}(a) shows that $K_\omega$ is a normal subgroup of $K_\sigma$ and 
part (b) of that lemma shows that $K_\sigma = K_0$.
 If $K_\omega \neq 1$, then $K$ is not regular, and so Lemma~\ref{solplinthsreg} shows that $K=[K,K]$. Thus part (2) of Definition~\ref{defn:trips} holds.

Suppose that $R$ is a normal subgroup of $K$ that is $G_\omega\mu$ invariant. Then $R$ is normalised by $KG_\omega=G$. Since $K$ is a plinth, we have that $R$ is intransitive, hence $K \neq RK_\omega$. Thus part (3) of Definition~\ref{defn:trips} holds and $(K,G_\omega\mu,K_\omega)$ is a semiprimitive triple as required.
\end{proof}

Below we detail a construction that takes as input a semiprimitive triple $(K,H,L)$ and produces a semiprimitive group. The group constructed will feature $K$ as a plinth and $H$  as a point-stabiliser. Condition (1)  guarantees that the group produced from a semiprimitive triple will be semiprimitive. Condition (2) encodes the point-stabiliser and   centraliser of a plinth and ensures that a plinth will act faithfully.  Condition (3)  guarantees that the group $K$ will be a plinth in the permutation group constructed.

\begin{construction}
\label{con:sptriples}
Let  $(K,H,L)$ be  a semiprimitive triple. We set 
$$X =   K \rtimes H$$
(with the action of $H$ on $K$ as automorphisms).
For convenience, we identify   $K$, $L$ and $H$ with their images in $X$. 
We write the elements of $X$ as tuples $( h, k)$ with   $k\in K$ and $h\in H$, with multiplication as below
$$(h,k)(h',k') = (hh',k^{h'}k').$$
Let 
$Y = HL$. We set 
$$ \Omega(K,H,L) = [X :Y ],$$
the set of right cosets of $Y$ in $X$,  and let $X$ act on $\Omega(K,H,L)$ by right multiplication. Set
$Z_0 = \{ (c_{u^{-1}},u) : u \in K_0 \}$, $Z=\{(c_{u^{-1}},u) \mid u \in L \}$ and let
$$\mathcal G(K,H,L) = X/Z.$$
\end{construction}

\begin{lem}
\label{lem:cent of K in X}
The centraliser of $K$ in $X$ is $Z_0$.
\end{lem}
\begin{proof}
It is easy to check that $Z_0 \leqslant C_X(K)$. Let $(1,k) \in K$ and  $(a,b) \in X$ be   arbitrary. Then $(1,k)^{(a,b) } = (1, (k^a)^b)$. Hence $(a,b) \in C_X(K)$ if and only if $k^{b^{-1}} = k^a$ for all $k\in K$. This gives $b \in K_0$ and $a=c_{b^{-1}}$, so that $(a,b) = (c_{b^{-1}},b) \in Z_0$. Thus $Z_0 = C_X(K)$.
\end{proof}

\begin{lem}
The kernel of the action of $X$ on $ \Omega(K,H,L)$ is $Z$.
\end{lem}
\begin{proof}
The kernel of the action of $X$ on $ \Omega(K,H,L)$ is $\mathrm{core}_X(Y)$. Note that $Z \leqslant Y$ and that $Z$ centralises $K$ since $Z \leqslant Z_0$. Moreover $H$ normalises $L$ so $H$ normalises $Z$. Thus $ Z \leqslant \mathrm{core}_X(Y)$. 

We observe that $K \cap Y = L$. Hence $ K \cap \mathrm{core}_X(Y)$ is a normal subgroup of $K$, contained in $L$. By Definition~\ref{defn:trips}(2) $L$ is core-free in $K$, hence $K \cap \mathrm{core}_X(Y)=1$. In particular, $\mathrm{core}_X(Y)$ centralises $K$. Now note that $Y = HL = HZ$ and so $\mathrm{core}_X(Y) = Z(\mathrm{core}_X(Y) \cap H)$. Thus $\mathrm{core}_X(Y) \cap H \leqslant C_H(K) = 1$, since Definition~\ref{defn:trips}(1) states that $H$ is a group of automorphisms of $K$. Hence $\mathrm{core}_X(Y) = Z$.
\end{proof}

\begin{lem}
The group $\mathcal G(K,H,L)$ is semiprimitive on the set $ \Omega(K,H,L)$ with point-stabiliser $HZ/Z \cong H$ and plinth $KZ/Z \cong K$. Moreover $KZ/Z \cap HZ/Z = L Z / Z \cong L$.
\end{lem}
\begin{proof}
First we note that $K \cong KZ/Z$, $Y = HZ$ and $H \cong HZ/Z$. Since $X=KY$, it follows that  $KZ/Z$ is transitive on $ \Omega(K,H,L)$ and since $Z$ centralises $K$, the actions of $H$ on $K$ and of $HZ/Z$ on $KZ/Z$ are the same. Since $HZ/Z = Y/Z$ is core-free in $X/Z$ (by definition of $X/Z$), Definition~\ref{defn:trips}(1) shows that the hypothesis of   Lemma~\ref{lem:g is semi prim} (with $X/Z$ in place of $G$, $KZ/Z$ in place of  $K$  and $Y/Z$ in place of $H$) is satisfied. Hence  $X/Z$ is semiprimitive on $ \Omega(K,H,L)$. 

We have that $KZ \cap Y = Z(K \cap Y) = ZL$ and $L \cap Z    \leqslant \mathrm{core}_K(L)=1$, so that $ZL/Z \cong L$.

Suppose that $R$ is a normal subgroup of $X/Z$ properly contained in $KZ/Z$. Then (by the second isomorphism theorem)  there is an $H$-invariant normal subgroup $R_0$  of $K$ such that $R_0Z/Z = R$. Now suppose that $R$ is transitive  on $ \Omega(K,H,L)$. Then $R(HZ/Z) = X/Z$. In particular, we have $KZ/Z = R(KZ/Z \cap HZ/Z) = R(LZ/Z)$. The second isomorphism theorem implies that $R_0L=K$, a contradiction to Definition~\ref{defn:trips}(3). Hence  $R$ is intransitive, and so $KZ/Z$ is a plinth of $X/Z$. 
\end{proof}

\begin{lem}
The centraliser of $KZ/Z$ in $X/Z$ is $Z_0/Z \cong K_0/L$.
\end{lem}
\begin{proof}
Since the natural quotient map restricts to an isomorphism between $K$ and $KZ/Z$, we have that $C_{X/Z}(KZ/Z) = C_X(K)Z/Z$. By Lemma~\ref{lem:cent of K in X} we have $C_X(K) = Z_0$. Hence $Z_0/Z = C_{X/Z}(KZ/Z)$. It is clear that $Z_0/Z$ is isomorphic to $K_0/L$ from the definition of $Z_0$ and $Z$.
\end{proof}
 
 We now show that there is an  equivalence between semiprimitive triples and semiprimitive groups with a plinth.
 
 \begin{thm}
Let $G \leqslant \mathrm{Sym}(\Omega)$ be a semiprimitive group with plinth $K$ and let $\omega \in \Omega$. Let $\mu : G_\omega \rightarrow \mathrm{Aut}(K)$ be the natural map induced by the conjugation action of $G_\omega$ on $K$. Then  $G$ is permutation isomorphic to $\mathcal G(K,G_\omega\mu,K_\omega)$.
 \end{thm}
 \begin{proof}
Write $H=G_\omega\mu$ and $L=K_\omega$ and continue with the notation established above. Without loss of generality, we assume that $\Omega = [G: G_\omega]$. Since $X=KY$ and $G=KG_\omega$, we may define the following
 \begin{eqnarray*}
 f: [X :Y ] \rightarrow \Omega & \text{by} & f : Y(1,k) \mapsto G_\omega k, \\
\pi : X \rightarrow G  & \text{by} & \pi : (h\mu,k) \mapsto hk.
\end{eqnarray*}
It is routine to check that $f$ is a bijection, that $\pi$ is a homomorphism with kernel $Z$ and that the pair $(f,\pi)$ is a permutational isomorphism. We simply note that $\pi$ is well-defined since $\mu$ is an isomorphism.
%
 \end{proof}
 
 We shall now illustrate a method of combining semiprimitive triples to produce new semiprimitive triples. First, for a semiprimitive triple $(K,H,L)$, note that the map $\tau : L \rightarrow \mathrm{Aut}(K)$ is a monomorphism. Indeed, if $x\in \ker(\tau)$ then $x \in \mathrm Z(K)$, so that $x\in \mathrm{core}_K(L)$. Hence $x=1$ by Definition~\ref{defn:trips}(2). Moreover, the following identity holds for all $x\in L$ and $h\in H$:
$$ (x^h)\tau = c_{x^h} = (c_x)^h = (x\tau)^h.$$
We shall use this identity in several places below.

\begin{defn}
\label{defn:product of trips}
Let $(K_1,H_1,L_1)$ and $(K_2,H_2,L_2)$ be semiprimitive triples. Let $\tau_i : L_i \rightarrow \mathrm{Aut}(K_i)$  be the natural maps induced by the conjugation action of $L_i$ on $K_i$.
Note also that $L_i\tau_i \leqslant H_i$ by Definition~\ref{defn:trips}(2).

Suppose that there exists an isomorphism $\mu : H_1 \rightarrow H_2$ such that 
$$L_1\tau_1 \mu  = L_2\tau_2.$$
We  define a product on such  triples by
$$(K_1,H_1,L_1) * (K_2,H_2,L_2) = (K_1 \times K_2, \mathrm{diag}_\mu(H_1,H_2), \mathrm{diag}_{\tau_1\mu\tau_2^{-1}}(L_1,L_2)).$$
\end{defn}

\begin{thm}
The  product of semiprimitive triples defined in Definition~\ref{defn:product of trips} is a semiprimitive triple.
\end{thm}
\begin{proof}
With the above notation, let $(K,H,L)=(K_1,H_1,L_1)*(K_2,H_2,L_2)$. Note that $H$ is a subgroup of $\mathrm{Aut}(K_1) \times \mathrm{Aut}(K_2)$ and so $H$ is indeed a group of automorphisms of $K$. We first check that Definition~\ref{defn:trips}(1) holds. Suppose that $U$ is a $H$-invariant normal subgroup of $K$ and that $B$ is the kernel of the action of $H$ on $K/U$. Then $B$ acts trivially on $K_1U/U \cong K_1/K_1 \cap U$. Since the action of $H$ on $K_1$ is the same as the action of $H_1$ on $K_1$, this says that $K_1 \cap U = K_1$. Thus $K/U =K_1K_2 /U \leqslant K_2 U/ U \cong K_2 / K_2 \cap U$. Similarly, the action of $H$ on $K_2$ is equivalent to the action of $H_2$ on $K_2$, and so we have that $K_2 \cap U= K_2$. Hence $U=K$ and part (1) of Definition~\ref{defn:trips} holds.

Recall the definition of $K_0$ from Definition~\ref{defn:trips}(2). We need to show that $L$ is a normal, core-free subgroup of $K_0$. We may write  $L= \mathrm{diag}_{\mu}(L_1\tau_1,L_2\tau_2^{-1})$. Since $L_1\tau_1$  is normal in $H_1$ and $L_2\tau_2^{-1}$ is normal in $H_2$ then, we find that $L$ is normal in $\diag_\mu(H_1,H_2)$.  For simplicity of notation, let us write $\sigma=\tau_1\mu\tau_2^{-1}$. Note that an element $(k_1,k_2) \in K$ is in $K_0$ if and only if there is $h\in H$ such that $h=(u,u\mu)$ and $x^{(k_1,k_2)}=x^h = x^{(u,u\mu)}$ for all $x \in K$. Now for $r = (t,t\sigma) \in L$ we have $r^{(k_1,k_2)} = (t,t\sigma)^{(u,u\mu)} = (t^u,(t\sigma)^{u\mu})$. Note that $t^u \in L_1$ since $H_1$ normalises $L_1$, hence we simply need to see that $(t^u)\sigma = (t\sigma)^{u\mu}$. Note that $(t^u)\tau_1 =(t\tau_1)^u$ (where the action of $u$ on the left is as an automorphism and the action on the right is as conjugation). Hence $(t^u)\tau_1\mu = ((t\tau_1)^u)\mu = (t\tau_1\mu)^{u\mu}$. Now since $t\tau_1\mu \in L_2 \tau_2$ and $u\mu \in H_2$, we have $((t \tau_1\mu)^{u\mu})\tau_2^{-1} = (t\tau_1\mu\tau_2^{-1})^{u\mu}=(t\sigma)^{u\mu}$. 

We now verify that $L$ is core-free in $K$. For $i=1,2$ let $\pi_i: K \rightarrow K_i$ be the natural projections. Let $R$ be the core of $L$ in $K$. Then $\pi_i(R)$ is normal in $K_i$. Since $\pi_i(L)=L_i$, we have that $\pi_i(R)=1$ for $i=1,2$. Hence $R \leqslant \ker(\pi_1) \cap \ker(\pi_2) = K_1 \cap K_2 = 1$.  Hence part (2) of Definition~\ref{defn:trips} holds.
 
We now verify that Definition~\ref{defn:trips}(3) holds.  Suppose  that $R$ is a normal $H$-invariant subgroup of $K$ such that $LR=K$. We may assume that $L \neq 1$, hence $L_1$ and $L_2$ are both non-trivial. By Definition~\ref{defn:trips}(2) we have that  each $K_i$ is perfect, and so $K$ is perfect. Now
$$K_i=\pi_i(K) = (LR)\pi_i = (L\pi_i)(R\pi_i) = L_i(R\pi_i).$$
Hence $R \pi_i = K_i$ for each $i=1,2$ by Definition~\ref{defn:trips}(3). Applying Lemma \ref{lem:normal subgroups of perfect groups} we have that $R=K$.  This completes the proof.
\end{proof}

For clarity, we give the interpretation of the previous result for permutation groups. 

\begin{cor}
\label{cor:product}
Let $G_1 \leqslant \mathrm{Sym}(\Omega_1)$ and $G_2 \leqslant \mathrm{Sym}(\Omega_2)$ be semiprimitive groups with plinths $K_1$ and $K_2$. Let $\omega_1 \in \Omega_1$ and $\omega_2 \in \Omega_2$ and suppose there is an isomorphism $\mu : (G_1)_{\omega_1} \rightarrow (G_2)_{\omega_2}$ such that $(K_1)_{\omega_1}\mu=(K_2)_{\omega_2}$.
Then there exists a semiprimitive group $G$ with plinth $K_1\times K_2$ and point-stabilisers isomorphic to $(G_1)_{\omega_1}$ such that $G_1$ and $G_2$ are quotient actions of $G$ via normal subgroups.
\end{cor}
%

\begin{defn}
We call the product of the two permutation groups $G_1$ and $G_2$ given in Corollary~\ref{cor:product} the \emph{glued product} of $G_1$ and $G_2$.
\end{defn}

\begin{rem}
Although the glued product of two semiprimitive groups is again a semiprimitive group, the glued product of any two innately transitive groups fails to be innately transitive -- every plinth of an innately transitive group is a minimal normal subgroup whereas  the construction of the glued product of two semiprimitive groups visibly gives a plinth which is not a minimal normal subgroup.
\end{rem}

\begin{eg}
Let $G_1=\mathrm{Sym}(3)=\mathrm D_6$ and $G_2=\mathrm D_{10}$. Then both $G_1$ and $G_2$ are semiprimitive with regular plinths (of order three and five respectively) and point-stabilisers of order two. Hence we may form the glued product, $(\mathrm C_3 \times \mathrm C_5) \rtimes \mathrm C_2$, which is isomorphic to $\mathrm D_{30}$.
\end{eg}

\begin{eg}
Let $G_1 = \mathrm{Sym}(7)$ and $H_1 = \mathrm{Alt}(5) \times \langle g \rangle$ where $g$ is a transposition. Let $G_2 = \mathrm{Alt}(5) \wr \mathrm C_2$ and $H_2 = \mathrm{diag}(\mathrm{Alt}(5),\mathrm{Alt}(5)) \rtimes \langle\sigma\rangle$ where $\sigma$ is an element of order two interchanging the normal copies of $\mathrm{Alt}(5)$. Then if $K_1$ and $K_2$ are plinths of $G_1$ and $G_2$, we have $K_1 \cap H_1 = \mathrm{Alt}(5)$ and $K_2 \cap H_2 = \mathrm{Alt}(5)$, and $H_1 \cong H_2$. Thus we may form the glued product of $G_1$ and $G_2$, which is isomorphic to
$$(\mathrm{Alt}(7) \times \mathrm{Alt}(5) \times \mathrm{Alt}(5) ) \rtimes \langle (g,\sigma) \rangle$$
with a point-stabiliser in $G$ equal to $\{(x,x,x) \mid x \in \mathrm{Alt}(5) \} \times  \langle (g,\sigma) \rangle.$
\end{eg}

We now give a sufficient condition that allows us to recognise glued products of semiprimitive groups.

\begin{thm}
\label{thm: decomp perm isom}
Suppose that $G$ is a semiprimitive group with plinth $K$ such that $K = K_1 \times K_2 $ is a $G$-invariant decomposition of $K$ with $K_1 \neq 1 \neq K_2$. Then $G$ is permutationally isomorphic to the glued product of the semiprimitive groups $G_1 = G/K_2$ and $G_2 = G /K_1 $.
\end{thm}
\begin{proof}
Since $K_1$ and $K_2$ are proper subgroups of $K$ that are normal in $G$, both are semiregular, hence $G_1$ and $G_2$ are semiprimitive. For $i=1,2$ let $\pi_i : G \rightarrow G_i$ be the canonical map (so that $\ker\pi_i = K_{3-i}$). For $i=1,2$ set $H_i = (G_\omega)\pi_i$ and $M_i = K\pi_i= K_i \pi_i$ so that we have  
$$G_i = (KG_\omega)\pi_i = (K)\pi_i (G_\omega)\pi_i = M_i  H_i.$$ 
Now $M_i$ is a plinth of $G_i$ by Lemma~\ref{lem:plinths in quotient actions}.  Since each $K_i$ is semiregular, the maps $\pi_i|_{G_\omega}$ are isomorphisms. Hence the map $\mu : H_1 \rightarrow H_2$ defined by $(h\pi_1)\mu = h\pi_2$  is an isomorphism (that is,  $\mu= (\pi_1|_{H_1})^{-1}\pi_2|_{G_\omega}$). The semiprimitive triples of $G_1$ and $G_2$ are $(M_1,H_1,M_1 \cap H_1)$ and $(M_2,H_2,M_2 \cap H_2)$ respectively, thus with $H=\diag_\mu(H_1,H_2)$ and $L=\diag_\mu(M_1\cap H_1,M_2 \cap H_2)$ we have
$$(M_1,H_1,M_1 \cap H_1)*(M_2,H_2,M_2 \cap H_2) = (M_1 \times M_2, H,L).$$
With the notation as in Construction~\ref{con:sptriples}, let $X=(M_1 \times M_2) \rtimes H$. We define a map $\sigma : X \rightarrow G$ by 
$$\sigma: (h\pi_1,h\pi_2,k_1\pi_1,k_2\pi_2 ) \rightarrow h k_1k_2  .$$
It is easy to verify $\sigma$ is a homomorphism, we now find $\ker(\sigma)$. 
Let $x=(h\pi_1,h\pi_2,k_1\pi_1,k_2\pi_2)$, and suppose $x\sigma = 1$. Then $hk_1k_2 = 1$, so that $k_1k_2 = h^{-1} \in K_\omega$. Further, $k_1 = h^{-1}k_2^{-1}$ so that $k_1\pi_1 = h^{-1}\pi_1$ and similarly $k_2\pi_2 = h^{-1}\pi_2$. Hence $x=(h\pi_1,h\pi_2,h^{-1}\pi_1,h^{-1}\pi_2) \in Z$ (with $Z$ as in Construction~\ref{con:sptriples}). Thus $\sigma$ is an isomorphism between $X/Z$ and $G$. Clearly the image of $Y$ is $G_\omega$, thus $G$ and $\mathcal G(M_1 \times M_2,H,L)$ are permutationally isomorphic.
\end{proof}

We now record some cases when it is impossible to glue innately transitive groups. In Section~\ref{sec:ex} we will show that, apart from these cases, all other glueings are possible.

\begin{lem}
\label{lem:non-glue}
The following pairs of innately transitive groups cannot be glued: (regular plinth, non-regular plinth),  (SD,CD), (HS,HC), (AS$^ \text{reg} $,DQ).
\end{lem}
\begin{proof}
Suppose that $G\leqslant \sym{\Omega}$ and $H \leqslant \sym{\Delta}$ are innately transitive with plinths $K$ and $L$ and let $\omega\in \Omega$ and $\delta \in \Delta$. A necessary condition to form the glued product of $G$ and $H$ is that there is an isomorphism $\mu : G_\omega \rightarrow H_\delta$ such that $(K_\omega)\mu =  L_\delta $. Thus clearly a product of a pair of types such as (regular plinth, non-regular plinth) is impossible. Suppose that $G$ has type SD and $H$ has type CD. Then there is a non-abelian simple group  $T$ such that $K_\omega \cong T$, whereas $L_\delta \cong S^\ell$ for some finite simple group $S$ and some integer $\ell>1$. A similar statement holds if $G$ has type HS and $H$ has type HC by considering $\mathrm{soc}(G_\omega)$ and $\mathrm{soc}(H_\delta)$. If $G$ has type AS$^{\text{reg}}$, then $G_\omega$ is soluble by the Schreier Conjecture, whereas if $H$ has type DQ, then $H_\delta$ involves a non-abelian simple group.
\end{proof}

\section{Structure Theorem}

We now prove a structure theorem for finite semiprimitive groups. Our division of the class of semiprimitive groups is based on the types of quotient actions that arise.  Examples of semiprimitive groups with the many different types of quotient actions are provided in Section~\ref{sec:ex}.

\begin{thm}
\label{thm:classification}
Let $G \leqslant \mathrm{Sym}(\Omega)$ be a finite semiprimitive group and let $\mathcal K$ be the set of plinths of $G$.
Then precisely one of the following holds:
\begin{enumerate}[(a)]
\item $\mathcal K=\{ K \}$. Let $  \mathcal S$ be a set of proper subgroups of $K$ that are maximal  with respect to being normal in $G$ and let $S = \bigcap_{M \in \mathcal S} M$. Then there exists a subset $\mathcal S' \subseteq \mathcal S$ such that  $G/S$ is permutationally isomorphic to the glued product of the innately transitive groups $G/M$ for $M\in \mathcal S'$  and one of the following holds:
\begin{enumerate}[(i)] 
\item $K$ is non-regular and each $G/M$ is of type AS$^{\text{non-reg}}$, ASQ$^\text{non-reg}$, PA, SD or CD;
\item  $K$ is regular and each $G/M$ is of type AS$^{\text{reg}}$, ASQ$^\text{reg}$, HA, TW, DQ or PQ;
\end{enumerate}
In cases (i) and (ii) respectively, $G$ cannot   have simultaneous  innately transitive quotients actions of types SD and CD, respectively, AS$^{\text{reg}}$ and DQ. Further, $\mathcal S = \mathcal S'$ unless case (ii) holds and there exists $M\in \mathcal S$ such that $G/M$ is of type HA.
\item $|\mathcal K | >1$ and each $K \in \mathcal K$ is regular. Let $ \mathcal S $ be a subset of $\mathcal K$ of size at least two and let $S=\bigcap_{K \in \mathcal S } K$. Then $G/S$ is permutationally isomorphic to the glued product of $|\mathcal S|-1$ primitive groups, all of type HS or all of type HC.   
\end{enumerate}
\end{thm}
\begin{proof}
Suppose first that $\mathcal K = \{ K\}$. Note that $K/M$ is a plinth of $G/M$ for each $M\in \mathcal S$ and that $K/M$ is a regular plinth if and only if $K$ is regular by Lemma~\ref{lem:plinths in quotient actions}. 
Hence, if $K$ is non-regular, then $G/M$ is an innately transitive group of type AS$^{\text{non-reg}}$, ASQ$^\text{non-reg}$, PA, SD or CD for each $M\in \mathcal S$, and if $K$ is regular, then $G/M$ is innately transitive of type AS$^{\text{reg}}$, ASQ$^\text{reg}$, HA, TW, DQ or PQ.
By Proposition~\ref{prop:duality}, $K/S = L_1 \times \dots \times L_r$ for some integer $r \leqslant |\mathcal S|$ and some minimal normal subgroups $L_1$, \dots, $L_r$ of $G/S$. 
Repeated application of Theorem~\ref{thm: decomp perm isom} shows that  $G/S$ is the glued product of the innately transitive groups $G/M$ for $M \in \mathcal S'$ with $\mathcal S' \subseteq \mathcal S$. By Lemma~\ref{lem:non-glue} it is impossible to form the glued product of SD and CD groups and of AS$^{\text{reg}}$ and DQ groups, hence $G/S$ cannot have simultaneous quotient actions of type SD and CD or of types AS$^{\text{reg}}$ and DQ in the respective cases. Finally, $K/S$ is perfect unless there is some $M\in \mathcal S$ such that $G/M$ is of type HA, and so $\mathcal S'=\mathcal S$ unless this occurs.

Suppose now that $|\mathcal K | >1$ and let $\mathcal K = \{ K , K_1, \dots,K_r \}$.  For $i=1,\dots,r$ set $M_i= K  \cap K_i$ so that $S = \bigcap_{i=1}^r M_i$. By Theorem~\ref{thm:mult plinths} we have $K /M_i$ is perfect for each $i$.  By Proposition~\ref{prop:duality}, we have that $ K/S = L_1 \times \dots \times L_s$ for some $s\leqslant r$ where each $L_i$ is a minimal normal subgroup of $G/S$ and $L_i$ is isomorphic to $K/ M_i$. In particular, $K /S$ is perfect, and so Proposition~\ref{prop:duality} shows that $s=r=|\mathcal S|-1$. Repeated application of Theorem~\ref{thm: decomp perm isom} shows that $G/S$ is the glued product of the permutation groups $G/M_i$ for $i=1,\dots,r$.  Theorem~\ref{thm:mult plinths} shows that $G/M_i$ is a primitive group of type HS or HC, and Lemma~\ref{lem:non-glue} shows that either all $G/M$ for $M\in \mathcal S$ are of type HS or all are of type HC. 
\end{proof}

In the proof of part (a)(ii) of the above theorem the CFSG is invoked. This is similar to the part of the proof of the O'Nan-Scott Theorem that  shows that an primitive group of almost simple type cannot have a regular socle.

We now prove Theorem~\ref{thm:rad g}. Recall the definition of $\mathrm{rad}(G)$ from Definition~\ref{defn:plinths}.
\bigskip

\emph{Proof of Theorem~\ref{thm:rad g}.}
Let $G$ be a semiprimitive group and apply Theorem~\ref{thm:classification} to $G$. If case (a) holds take $\mathcal S$ to be the set of all proper   subgroups of $K$ that are maximal  with respect to being normal in $G$ so that $S = \mathrm{rad}(G)$. If case (b) holds then take $\mathcal S = \mathcal K$.
\qed

\section{Examples of semiprimitive groups}
\label{sec:ex}
  We now give examples of semiprimitive groups with quotient actions of all the types listed in cases (a)(i), (a)(ii) and (b) of Theorem~\ref{thm:classification}. All groups considered in this section will therefore be finite. Most of the examples below are constructed by applying Corollary~\ref{cor:product} to a set of semiprimitive groups.

\begin{eg}
Let $H= \mathrm{Alt}(5)\times \langle x \rangle$ where $x$ has order 2. 

Let $G_1 = \mathrm{Sym}(7)$ and identify $H$ with the obvious subgroup of the centraliser of a transposition. Then $G_1$ on the set of cosets of $H$ is quasiprimitive of type AS$^{\text{non-reg}}$. 

Let $G_2 = \mathrm{Alt}(9) \rtimes \langle \sigma \rangle$ where $\sigma$ is an inner automorphism induced by a double transposition. Identify $H$ with a subgroup of $G_2$ where $x$ is identified with $\sigma$. Then on the set of cosets of $H$, $G_2$ is innately transitive of type ASQ$^{\text{non-reg}}$. 

Let $G_3 = \mathrm{Alt}(6) \wr C_2$, and identify $H$ with the normaliser of a diagonal $\mathrm{Alt}(5)$ subgroup of $\mathrm{Alt}(6)\times \mathrm{Alt}(6)$. Then $G_3$ is quasiprimitive of type PA. 

Finally, let $G_4 = \mathrm{Alt}(5) \wr C_2$ and identify $H$ with the normaliser of a diagonal subgroup of  $\mathrm{Alt}(5)\times \mathrm{Alt}(5)$. Then $G_4$ is primitive of type SD on the set of cosets of $H$ in $G_4$.

We have identified $H$ with a subgroup of $G_1$, $G_2$, $G_3$ and $G_4$. If $K_i$ is a plinth of $G_i$ for $i=1,2,3,4$, then we observe that $H \cap K_i = \mathrm{Alt}(5)$. Corollary~\ref{cor:product} shows there is a semiprimitive group $G$ with quotient actions $G_1$, \dots, $G_4$. Thus $G$ is an example of a group occurring in case (a)(i) of Theorem~\ref{thm:classification} with a quotient action of type SD.
\end{eg}

\begin{eg}
Let $H = (\mathrm{Alt}(5) \times \mathrm{Alt}(5) ) \rtimes \langle \sigma, \tau\rangle$ where $\sigma$ and $\tau$ have order two, $\sigma$ switches the two copies of $\mathrm{Alt}(5)$ in $H$ and $\tau$ is central in $H$.

Let $G_1 = \mathrm{Alt}(5)^4 \rtimes \langle \sigma, \tau \rangle $ where $\langle \sigma, \tau \rangle$ acts on the four copies of $\mathrm{Alt}(5)$ as the Klein 4-group. Then we may identify $H$ with a subgroup of $G_1$, where each $\alt{5}$ is a diagonal subgroup across two copies of $\alt{5}$. On the set of cosets of $H$, $G_1$ is a primitive group of type CD.

Let $G_2= \mathrm{Alt}(6)^4 \rtimes \langle \sigma, \tau \rangle $ where $\langle \sigma, \tau \rangle$ acts on the four copies of $\mathrm{Alt}(6)$ as the Klein 4-group. Then $G_2$ contains a copy of $H$ (an embedding similar to that in $G_1$), and on the set of cosets of $H$, $G_2$ is a quasiprimitive group of type PA.

Let $G_3 = \mathrm{Sym}(12) \rtimes \langle \mu \rangle$ where $\mu$ acts on $\mathrm{Sym}(12) $ as conjugation by the element $(1,6)(2,7)(3,8)(4,9)(5,10)$. We may identify $H$ with a subgroup of $G_3$ such that $\sigma$ is identified with $\mu$ and $\tau$ is identified as the involution $(11,12)$. Then the action of $G_3$ on the set of cosets of $H$ in $G_3$ is innately transitive of type ASQ$^{\text{non-reg}}$.

Let $G_4 = \mathrm{PSL}_{10}(p) \rtimes \langle \mu, \alpha \rangle$ where $p\equiv 3 \mod 4$, $\mu$ is the inverse transpose automorphism and $\alpha$ is a diagonal automorphism induced by the matrix $$ \left [ \begin{array}{cc} 0 & \mathrm I_5 \\ \mathrm I_5 & 0 \end{array} \right ]$$
(which is non-inner because of the conditions on $p$).
We identify $\mathrm{Alt}(5)$ with a subgroup $T$ of $\mathrm{GL}_5(p)$ via the permutation representation on 5 points, and then $\mathrm{Alt}(5) \times \mathrm{Alt}(5)$ may be identified with the subgroup $T \times T^\alpha$ of $\mathrm{GL}_{10}(p)$. Note that both $T$ and $T^\alpha$ are centralised by $\mu$ since $T$ preserves an orthogonal form on the permutation module.  Taking the image of $T\times T^\alpha$ in $G_4$ we may identify $H$ with the subgroup $(T \times T^\alpha)\rtimes \langle \mu,\alpha \rangle$. The action of $G_4$ on the set of cosets of $H$ is quasiprimitive of type AS$^{\text{non-reg}}$.

For $i=1,2,3,4$ let $K_i$ be a plinth of $G_i$. Then $K_1 \cong  \mathrm{Alt}(5)^4$, $K_2 \cong \mathrm{Alt}(6)^4$, $K_3 \cong \mathrm{Alt}(12)$ and $K_4 \cong \mathrm{PSL}_{10}(p)$. In each case, we have $K_i \cap H = \mathrm{Alt}(5)^2$. Corollary~\ref{cor:product} shows that there is a semiprimitive group with quotient actions $G_1$, $G_2$, $G_3$ and $G_4$. Thus $G$ is an example of a group occurring in case (a)(i) of  Theorem~\ref{thm:classification} with a quotient action of type CD.
\end{eg}

\begin{eg}
Let $H=\mathrm C_4$ and let $T=\mathrm{PSL}(2,7^4)$. Let $G_1 =  T \rtimes H$ with $H$ acting as field automorphisms (AS$^{\text{reg}}$ type). Let $G_2 = T \rtimes H$ with $H$ an inner automorphism (ASQ$^{\text{reg}}$ type). Let $G_3 =   \mathrm C_5 \rtimes H = \mathrm{Frob}_{20}$ (HA type).  Let $G_4=T^4 \rtimes H = T \wr H$ (TW type), $G_5 = T^2 \rtimes H$ where  $H$ is generated by  $c_{(1,t)}\sigma$, for some element $t\in T$ of order order two and $\sigma$ an element interchanging the two copies of $T$. Note that $G_5$ is innately transitive of PQ type since  the centraliser of a plinth has order two (equal to $\langle (t,t,(c_{(t,t)})^{-1}) \rangle$). 

Corollary~\ref{cor:product} shows there is a semiprimitive group $G$ with quotient actions $G_1$, \dots, $G_5$, that is, having quotient actions of type AS$^{\text{reg}}$, ASQ$^{\text{reg}}$, HA, TW and PQ. Thus $G$ is an example of a group appearing in case (a)(ii) of Theorem~\ref{thm:classification}.
\end{eg}

\begin{eg}
Let $T=\mathrm{Alt}(6)$ and let $H=\mathrm C_2 \times \mathrm{Alt}(5)$. Let  $G_1 = \mathrm{Alt}(9)\rtimes H$ with $H$ acting as inner automorphisms (ASQ$^{\text{reg}}$ type). Let 
$G_2 = 3^4 \rtimes H$ where $3^4$ is an irreducible module for $H$ (HA type). 
Let  $G_3 =T\wr H$ (TW type). Let  
$G_4 = T^5 \rtimes( \alt{5} \times \mathrm C_2) \cong (T \wr \alt{5})\rtimes \mathrm C_2 $ with $C_2$ acting diagonally as an inner automorphism of order two and $\alt{5}$ permuting the copies of $T$. Then $G_4$ is innately transitive of PQ  type since the centraliser of the plinth $T^5$ has order two. Let 
$G_5 = (\mathrm{Alt}(5) \wr \mathrm C_2) \rtimes \alt{5}$ with $\mathrm C_2$ permuting the copies of $\alt{5}$ and  $\alt{5}$ acting diagonally as inner automorphisms (DQ type).

Now Corollary~\ref{cor:product} shows there is a semiprimitive group $G$ with point-stabiliser isomorphic to  $H$ with quotient actions of type ASQ$^{\text{reg}}$, HA, TW, PQ and DQ. Hence $G$ is an example of a group appearing in case (a)(ii) of Theorem~\ref{thm:classification}.
\end{eg}

\begin{eg}
Let $T$ be a non-abelian simple group and let $\ell$ and $n$ be integers with $n>1$ and $\ell \geqslant 1$. Let $S=T^\ell$ and let $G = S^n \rtimes \mathrm{Sym}(\ell)=(S_1 \times \dots \times S_n)\rtimes \mathrm{Sym}(\ell)$, where the $\mathrm{Sym}(\ell)$ subgroup permutes the $\ell$ copies of $T$ in each copy of $S$. View $H = T^\ell \rtimes \mathrm{Sym}(\ell)$ as a subgroup of $G$ by identifying the i$^{\text{th}}$ copies of $T$ from  each copy of $S$. Let $G$ act on the set of cosets of $H$, note that $H$ is core-free in $G$. Let $K_i = \prod _{j \neq i } S_j$. Then each $K_i$ is a regular plinth of $G$ and for any $i\neq j$ we have $K_i \cap K_j \cong T^{(n-2)\ell}$. Moreover, $G/K_i \cap K_j$ is primitive of type HS if $\ell = 1$ and HC if $\ell>1$. Hence $G$ is an example of a group occurring in case (b) of Theorem~\ref{thm:classification}. 
\end{eg}

The following example shows that groups of type appearing in case (b) of Theorem~\ref{thm:classification} can also have other types of quotient actions.

\begin{eg}
Let $T=\alt{6}$ and for $i=1,2$ let  $S_i=\alt{5}$. Let $S_2$ act on $T$ and $S_1$ as inner automorphisms and set $G=(T \times S_1) \rtimes S_2$. Let $G$ act on  the set of cosets of $ S_2$. Then the subgroup $K_1 = TS_1$ is a regular normal subgroup. For a given copy $R$ of $\alt{5}$ in $T$, let  $S_3 = \{ (x,x,c_{x^{-1}}) : x \in R \}$. Then $S_3$ is a normal intransitive subgroup of $G$. Let $K_2 = TS_3$. Then $K_2$ is a normal regular subgroup of $G$. Moreover, $ T = K_1 \cap K_2$ and $G/T$ is primitive of type $HS$. Thus $G$ is an example of a group occurring in case (b) of Theorem~\ref{thm:classification}. Note also that $G/S_1$ is innately transitive of type ASQ$^{\text{reg}}$.
\end{eg}

\section{Wildness}

In this section we  give examples of semiprimitive groups which might be considered as evidence that semiprimitive groups are ``wild''. The first example shows that there is no control over the composition factors in a plinth of a semiprimitive group.

\begin{eg} 
Let $G$ be a finite semiprimitive with plinth $K$, point-stabiliser $H$ and let $n$ be the degree of $G$.
Let $T$ be a finite non-abelian simple group and let $G^T = T \wr KH$. Identifying $H$ and $K$ with their images in $G^T$, we claim that $G^T$ is semiprimitive on the set of cosets of $H$ with plinth $T^nK$.

Indeed, since $G$ acts transitively on the factors of $T^n$, we have that every non-trivial normal subgroup of $G^T$ contained in $T^nK$ must contain $T^n$. Hence the non-trivial normal subgroups of $G^T$ contained in $T^nK$ are in bijection with the normal subgroups of $K$. In particular, each proper non-trivial normal subgroup of $T^nK$ is intransitive (pass to the quotient $G^T/T^n \cong G$ to see this). Thus $T^nK$ is a plinth of $G^T$. Moreover,   the action of $H$ on $T^nK$ is faithful since $H$ acts faithfully on $T^n$, and $H$ acts faithfully on each $H$-invariant quotient of $T^nK/T^n$ since they correspond to $H$-invariant quotients of $K$. Thus Lemma~\ref{lem:g is semi prim} shows that $G^T$ acting on the set of cosets of $H$ is semiprimitive.

For a sequence of non-abelian simple groups $(T_1,T_2,\dots)$ let $G_1 = G^{T_1}$ and for $i\geqslant 2$ let 
$G_i = (G_{i-1})^{T_i}$, by the above paragraph each $G_i$ is a semiprimitive group. Thus for any finite set $\mathcal{S}$ of non-abelian simple groups, there exists a semiprimitive group $R$ with a plinth $K$ such that each group in $\mathcal S$ appears in a composition series of $K$.
\end{eg}

We now give an example which shows that there is no control over the structure of normal semiregular subgroups outside the plinths of semiprimitive groups, and in fact, that this is the case even for innately transitive groups.

\begin{eg} 
Let $M$ be any finite group. Pick an integer $n \geqslant 5$ such that $M \times V \leqslant \mathrm{Alt}(n)$ for some non-trivial subgroup $V \leqslant \mathrm{Alt}(n)$. Set $G=\mathrm{Alt}(n) \times M$ and let 
$$H=\{ (mv,m) : m \in M, v\in V \}.$$
Note that $G=\mathrm{Alt}(n) H$ and that $H$ is core-free in $G$. Moreover, since $\mathrm{Alt}(n)$ is simple, $\mathrm{Alt}(n)$ is a plinth of $G$, and $G$ is innately transitive of type ASQ$^{\text{reg}}$ if $V=1$ and ASQ$^{\text{non-reg}}$ otherwise. Thus $G$ is semiprimitive if $H$ acts faithfully on $\mathrm{Alt}(n)$. Since the kernel of the action is contained in $\mathrm Z(\mathrm{Alt}(n))=1$, we have that $G$ is semiprimitive. Moreover, $M$ is a normal semiregular subgroup of $G$.
\end{eg}

 Lemma~\ref{solplinthsreg} shows that a non-regular plinth of a semiprimitive group must be perfect. The following example seeks to address the converse to this statement: is every perfect group (isomorphic to) a non-regular plinth in some semiprimitive group?
 
\begin{eg}
\label{eg:cfperfect}
Let $K$ be a finite centre-free perfect group and let $S$ be the largest semisimple quotient of $K$ (that is, the quotient of $K$ by the smallest normal subgroup $R$ of $K$ such that $K/R$ is a direct product of non-abelian simple groups). Then $S$ is a direct product of non-abelian finite simple groups. Pick $h\in K$ of  order a power of two such that $h$ projects to an   involution in each non-abelian finite group $T$ that is a direct factor of $S$. Such an element acts faithfully on each quotient of $S$.

Since $K$ is centre-free, $H:=\langle h\rangle$ is core-free in $K$ (otherwise the unique involution in $H$ would be in the centre of $K$). Let $K$ act on the set of cosets of $H$. By Lemma~\ref{lem:sp iff h faithful on k/y} this action is semiprimitive: if $N$ is a  normal subgroup of $K$ such that $H$ does not act faithfully on $K/N$, then $H$ centralises $K/N$ and so $H$ centralises $K/M$ where $M$ is a maximally normal subgroup of $K$ containing $N$. In particular, $K/M$ is a quotient of $T$ on which $h$ does not act faithfully, a contradiction to our choice of $h$. 
%
%
\end{eg}

\section{Graph-theoretical problems}
\label{sec:graphs}

The authors' interest in semiprimitive groups is mostly due to  \cite{psv}, and we are thus motivated to explore further  graph-theoretical problems. Let $\Gamma$ be a locally finite graph and let $G\leqslant \mathrm{Aut}(\Gamma)$ be vertex-transitive. We say that $(\Gamma,G)$ is locally $L$ (locally $\mathcal P$) for a permutation group $L$ (property $\mathcal P$ of permutation groups) if  for each vertex $x\in \Gamma$ we have $G_x^{\Gamma(x)}$ is permutation isomorphic to $ L$ ($G_x^{\Gamma(x)}$ has property $\mathcal P$).  Here $G_x^{\Gamma(x)}$ is the permutation group induced by $G_x$
 on the neighbourhood $\Gamma(x)$ of $x$ in $\Gamma$.

The following result shows that the automorphism groups of graphs belonging to a large family are semiprimitive. The proof is from \cite[Lemma 1.6]{cheryl-imp sym}. 

\begin{lem}
\label{lem:graph problem}
Let $\Gamma$ be a connected non-bipartite graph and let $G \leqslant \mathrm{Aut}(\Gamma)$. Suppose that $G$ is locally quasiprimitive and vertex-transitive. Then $G$ is semiprimitive on $V\Gamma$.
\end{lem}
\begin{proof}
Suppose that $N$ is a normal subgroup of $G$ that is not semiregular. Then $N_x \neq 1$ for some $x\in \Gamma$. By connectivity of $\Gamma$, we have that $N_x^{\Gamma(x)} \neq 1$. Since $G_x^{\Gamma(x)}$ is quasiprimitive, and $N_x^{\Gamma(x)}$ is a non-trivial normal subgroup, we have that $N$ is locally-transitive. Then $N$ has at most two orbits on $V\Gamma$. Since $\Gamma$ is non-bipartite, $N$ has exactly one orbit, and so $N$ is transitive. Hence $G$ is semiprimitive.
\end{proof}

Let $T$ be a non-abelian finite simple group with a Sylow 2-subgroup $S$ and let $G$ be the permutation representation of  $T$ acting on the set of right cosets of $S$. Then $G$ is quasiprimitive. Suppose that  $G$ is (permutationally isomorphic to) a  vertex-transitive group of automorphisms of a connected locally quasiprimitive graph, $\Gamma$ say. Since the local action is at the same time quasiprimitive, and induced by $S$, a 2-group, the local action must be cyclic of order two. Thus $\Gamma$ must be a cycle, and hence $G$ cannot act faithfully. Thus the   above lemma has no converse.

A finite permutation group $L$ is said to be \emph{graph-restrictive}  \cite{verret} if there exists a constant $c=c(L)$ such that for every locally $L$ pair $(\Gamma,G)$ we have $|G_x|\leqslant c$. In this language, a conjecture of  Weiss \cite{weissc} states that every finite primitive permutation group is graph-restrictive. The conjecture was  generalised by Praeger \cite{praegerc},  replacing primitive by quasiprimitive.   Poto\v{c}nik, Spiga and Verret \cite{psv} have conjectured that a finite permutation group is graph-restrictive if and only if it is semiprimitive and have shown that every graph-restrictive group is semiprimitive. Below we show that this conjecture is true for semiprimitive groups of the type appearing in case (b) of Theorem~\ref{thm:classification}. Recall that for a prime $p$ and a finite group $X$,   $\mathrm O_p(X)$ is the largest normal $p$-subgroup of $X$.

\begin{lem}
\label{lem:two plinths pcore 1}
Suppose that $L \leqslant \mathrm{Sym}(\Omega)$ is a finite semiprimitive group with at least two plinths and let $\omega \in \Omega$. Then for all primes $p$ we have $\mathrm O_p(L_\omega)=1$.
\end{lem}
\begin{proof}
Let $K$ and $R$ be plinths of $L$ and let $\ba{L} = L / K \cap R$. Let $\Delta$ be the set of $(K \cap R)$-orbits so that $\ba{L}$ acts primitively on $\Delta$ of type HS or HC by Theorem~\ref{thm:mult plinths}. For $\delta \in \Delta$,  Lemma~\ref{lem:prim of type hs hc} shows that   $\mathrm O_p(\ba{L}_\delta) =1$ for all primes $p$. For $\omega \in \delta$ we have  $\ba{L}_\delta = \ba{(K\cap R) L_\omega}=\ba{L_\omega}$. Since $K\cap R$ is semiregular, Lemma~\ref{lem: quotient of sp is sp} shows that  $L_\omega \cong \ba{L_\omega}$  and we are done.
\end{proof}

For a finite group $X$, the generalised Fitting subgroup, $F^*(X)$, is the product of the layer of $X$ (defined before Lemma~\ref{lem:non-perfect plinth subgroup}) and the Fitting subgroup of $X$ (the largest normal nilpotent subgroup). We refer the reader to \cite[Chapter 9]{Isaacs} for properties that we use below.

\bigskip

\emph{Proof of Theorem~\ref{thm:intro-grp}.}
Let $L \leqslant \mathrm{Sym}(\Omega)$ be a finite semiprimitive group with at least two plinths. 
Let $(\Gamma,G)$ be a locally $L$ pair and let $\{x,y\}$ be an edge of $\Gamma$. We prove that $G_{xy}^{[1]}=1$. Assume for a contradiction that this is false. Then \cite[Corollary 2]{pablo} shows that there is a prime $p$ such that $G_{xy}^{[1]}$  and $F^*(G_{xy})$ non-trivial are $p$-groups. If $F^*(G_{xy}) \leqslant G_x^{[1]} $, then we have that $F^*(G_{xy}) = G_x^{[1]}$ and so \cite[Lemma 3.1(a)]{lukemichael3n2} would imply that $F^*(G_{xy})=1$, a contradiction. Hence $\mathrm O_p(G_{xy}^{\Gamma(x)}) \neq 1$. On the other hand, $G_x^{\Gamma(x)}$ is permutationally isomorphic to $L$, and so Lemma~\ref{lem:two plinths pcore 1} implies $\mathrm O_p(G_{xy}^{\Gamma(x)})=1$, a contradiction. Hence $G_{xy}^{[1]}=1$ as required. \hfill \qed

\section{Semiprimitive groups from wreath products}

\label{sec:wreath}

Standard constructions of permutation groups come from wreath products, in either the imprimitive action or the product action. In this section we seek to determine necessary and sufficient conditions for a wreath product to be semiprimitive. First we set out our notation.

Let $\Delta$ and $I$ be (possibly infinite) sets. We view $\Delta^I$ as the set of functions $f: I \rightarrow  \Delta$. For $\delta \in \Delta$ we denote by
$\res_\delta(\Delta^I)$
the set of functions $f$ such that $f(i) = \delta$ for all but finitely many elements  $i\in I$.

Let $T$ be a \emph{non-trivial} permutation group on $I$. Then $T$ acts naturally on $\Delta^I$ via $f^t(i):=f(it^{-1})$. For $t\in T$ we define  $\mathrm{supp}(t):=\{i \in I \mid it \neq i \}$. The subgroup consisting of elements of finite support is the group $\res(T) = \{ t\in T \mid |\mathrm{supp}(t)| < \infty \}$. In fact $\res(T)$ is a normal subgroup of $T$ (equal to $T$ if $I$ is finite).

If $M$ is a group  then $M^I$ (the set of functions $f: I \rightarrow M$) acquires the structure of a group via $(fg)(i):=f(i)g(i)$. We define $\res(M^I) := \res _1(M^I)$ and note that $\res(M^I)$ is a normal subgroup of $M^I$. For a subgroup $R$  of $M^I$ and $i\in I$ we set $R(i) = \{ r(i) \mid r \in R \}$. Note that $R(i)$ is a subgroup of $M$.

Suppose now that $M$ is a \emph{transitive} subgroup of $\sym{\Delta}$. Then $M^I$ acts on $\Delta^I$ via 
$$f^h(i): = f(i)^{h(i)}\quad  h\in M^I,\ f\in \Delta^I, \  i\in I.$$
 In this action, for $t\in T$, we have $f^{t^{-1} h t } = f^{h^t}$. Thus the groups $\res(M^I)$ and $M^I$ are  normalised by $T \leqslant \sym{\Delta^I}$ (as a subgroup of $\sym{\Delta^I}$). The \emph{unrestricted wreath product} is:
$$ M \Wr T : = \langle M^I, T \rangle = M^I \rtimes T \leqslant \sym{\Delta^I}.$$

Fix $\delta \in \Delta$ and define $f_\delta \in \Delta^I$ by $f_\delta(i)= \delta$ for all $i\in I$. Now $ \res_\delta(\Delta^I)$ is the orbit of $\res(M^I)$ containing $f_\delta$. The \emph{restricted wreath product} is:
$$ M \wWr T : = \langle \res(M^I), T \rangle = \res(M^I) \rtimes T \leqslant \sym{\res_\delta(\Delta^I)}.$$

For $J \subseteq I$ we define 
\begin{eqnarray*}M_J  & = &  \{ f\in M^I \mid f(i)=1 \ \text{ for all } \ i \notin J\}.
\end{eqnarray*}
 If $J=\{i\}$ we  write  $M_i$ in place of  $M_J$. Note that for finite subsets $J \subseteq I$ we have $M_J \leqslant \res(M^I)$ and if $J$ and $J'$ are two disjoint subsets of $I$ then $[M_J,M_{J'}]=1$.


The following result  is folklore.

\begin{lem}
\label{lem:normal subgroups of perfect groups}
Suppose that $M$ is a perfect  group. If $D$ is a normal subgroup of $\res(M^I)$ such that $D(i) = M$ for each $i \in I$, then  $D=\res(M^I)$.
\end{lem}
\begin{proof}
For $m\in M$ let $f_m\in M_i$ be such that $f_m(i)=m$. Now for each $c\in M$ there is $d\in D$ such that $d(i)=c$. For $j\neq i$ we have  $[d,f_m](j) = 1$ and $[d,f_m](i)= f_{[c,m]}=[f_c,f_m]$. Since $M \cong M_i$ we have $M_i = [M_i,M_i]$ and so $M_i \leqslant D$. In particular, $\res(M^I) \leqslant D$, and so $D=\res(M^I)$.
\end{proof}

\begin{lem}
If $M \Wr T$ or $M \wWr T$ is semiprimitive, then  $M$ is semiprimitive on $\Delta$.
\end{lem}
\begin{proof}
Suppose that $M$ is not semiprimitive and let $N$ be an intransitive non-semiregular normal subgroup of $M$. We claim that the normal subgroups $\res(N^I)$ and $N^I$ are both intransitive and non-semiregular.

Since $M$ is transitive and $N$ is normal and non-semiregular, we have   $N_\delta \neq 1$. Now 
$1 \neq \res( (N_\delta ) ^I) \leqslant \res(N^I)_{f_\delta}$ so both $\res(N^I)$ and $N^I$ are non-semiregular. Further, pick $\beta \notin \delta^N$ and $i\in I$. Define $f \in \Delta^I$ by $f(j)= \delta$ for $j\neq i$ and $f(i)=\beta$. Then $f\in \res_\delta(\Delta^I)$ and $f_\delta$ and $f$ lie in different $N^I$-orbits. Hence both $\res(N^I)$ and $N^I$ are intransitive.
\end{proof}

We now focus on the case where $M$ is semiprimitive. The first class of groups to deal with is the class of regular groups. Set $G=M \wWr T$.

\begin{lem}
Suppose that  $M$ is regular and perfect. Then $G$ is semiprimitive if and only if  $T$ acts faithfully on each $T$-orbit on $I$.
\end{lem}
\begin{proof}
Since $M$ is regular, we have  $G_{f_\delta} = T$ and $\res(M^I)$ is a regular plinth of $G$. Thus, by Lemma~\ref{lem:sp iff h faithful on k/y}, $G$ is semiprimitive if and only if $[\res(M^I),S]=\res(M^I)$ for each non-trivial normal subgroup $S$ of $T$. 

Suppose first that $J \subseteq I$ is a $T$-orbit and that  $S$, the kernel of the action of $T$ on $J$,  is non-trivial. Note $J \neq I$ since $T$ is non-trivial and acts faithfully on $I$.
 We calculate that for $f\in M^I$ and $t\in S$ we have $[f,t] \in M_J$. Hence
$[M^I,S] \leqslant  M_J < M^I$.
This yields $[\res(M^I),S] \leqslant M_J \cap \res(M^I) < \res(M^I)$. Thus Lemma~\ref{lem:sp iff h faithful on k/y} shows that $G$ is not semiprimitive.

Suppose now that $T$ is orbit faithful and let  $N$ be a normal subgroup of $G$ that is not semiregular. Since $M$ is regular,  $R:=N \cap G_{f_\delta} = N \cap T \neq 1$. Let $J$ be the set of fixed points of $R$ on $I$. Since $R$ is a normal subgroup of $T$,  $J$ is a union of $T$-orbits and  $R$ acts trivially on each $T$-orbit on $J$. Since $T$ is orbit faithful, we conclude that $J= \emptyset$. Let $i\in I$ be arbitrary. Then there is $\sigma \in R$ such that $i\sigma= j \neq i$. Let $m,m' \in M_i$ and note that $x:=[[m,m'],\sigma] \in N$. Now $x=[m,m']^{-1}[g,g']$ where $g,g'\in M_j$ are the images of $m$ and $m'$ under conjugation by $\sigma$. Now let $g'' \in M_j$ be arbitrary. Then $[x,g''] = [g,g',g'']$ (since $M_i$ and $M_j$ commute) and so $[g,g',g''] \in N$ since $N$ is normal in $G$. Hence $[M_j,M_j,M_j] \leqslant N$. Since $M_j \cong M$ is perfect, we have  $[M_j,M_j,M_j] = M_j$. Hence $M_j \leqslant N$ for all $j \in i^R$.  Since our choice of $i$ was arbitrary, we have $\res(M^I) \leqslant N$ and thus $N$ is transitive, as required.
\end{proof}

\begin{lem}
Suppose that $M$ is regular and not perfect. Then $G$ is not semiprimitive.
\end{lem}
\begin{proof}
Since $M$ is regular, $\res(M^I)$ is a plinth of $G$.  Since $M$ is not perfect, there is a proper normal subgroup $D$ such that $M/D$ is abelian. Then $\res(D^I)$ is normal in $G$ and $G/ \res(D^I) = \res(M^I)/\res(D^I) \rtimes T \cong  \res( (M/D)^I) \rtimes T$. 

Now  let $\overline{W} = \{ f\in \res( (M/D)^I) \mid \prod_{i\in I} f(i)=1\}$. (Note that $\overline{W}$ is well-defined since elements of $\res((M/D)^I)$ have finite support.) Clearly $\overline{W}$ is $T$-invariant, and since $M/D$ is abelian, we have $[\res((M/D)^I),T] \leqslant \overline{W}$. Let $W$ be the preimage in $M \wWr T$ of $\overline{W}$. Then $[\res(M^I),G_{f_\delta}] \leqslant W$. Hence $G$ is not semiprimitive by Lemma~\ref{lem:sp iff h faithful on k/y}.
%
\end{proof}

\begin{lem}
Suppose that $M$ is non-regular and that $T$ is intransitive. Then $G$ and $M \Wr T$ are not semiprimitive.
\end{lem}
\begin{proof}
Suppose that $J$ is an orbit of $T$. Then the subgroup $\res(M_J)$ is normal in both $G$ and $M \Wr T$. Since $M$ is not regular on $\Delta$,  this subgroup is not semiregular on $\Omega$, and so if $J \neq I$, then $G$ and $M \Wr T$ contain  an intransitive normal subgroup which is not semiregular. Hence $G$ and $M \Wr T$ are not semiprimitive.
\end{proof}

\begin{lem}
Suppose that $M$ is semiprimitive and  non-regular and suppose that $T$ is transitive. Then $G$ is semiprimitive. 
\end{lem}
\begin{proof}
Let $N$ be a normal subgroup of $G$ that is non-semiregular. We aim to show that $\res(D^I) \leqslant N$, and thereby prove that $N$ is transitive. 

Let $\omega = f_\delta$. Then $S:= N \cap G_\omega \neq 1$. Since $G_\omega = M_\delta \wWr T$, we have that $R:= S \cap \res((M_\delta)^I) \neq 1$. 
 Since $T$ is transitive, for any $i,j \in I$ we have that $R(i)=R(j) = R_0$ for some non-trivial normal subgroup $R_0$ of $M_\delta$.  Set $D=[R_0,M]$ and note that $D$ is a transitive normal  subgroup of $M$ by Lemma~\ref{lem:[g,r] trans}.

By the normality of $N$ in $G$, we have that $N$ contains $[\res(M^I), N]$ and therefore contains $[\res(M^I),R]$. Let $i\in I$ and $m\in M$ and define $f_m\in M_i$ such that $f_m(i)=m$, note that $f_m\in \res(M^I)$. Now for any $r\in R_0$ we have that there is  $f \in R$ such that $f(i) = r$. Now
$$[ f_m , f]= f_{[m,f(i)]} = f_{[m,r]}.$$
Thus $[M_i,R](i)$ contains $D$ and so $[\res(M^I), R ] \geqslant \res(D^I)$. Since $D$ is a transitive normal subgroup of $M$ we have that $\res(D^I)$ is transitive on $\res_\delta(\Delta^I)$.
Thus $[\res(M^I),R]$ is transitive as required.
%
\end{proof}

In summary, for the restricted wreath product we have:

\begin{thm}
\label{thm:wreath products}
The product action of $G=M \wWr T$ is semiprimitive if and only if either
\begin{itemize}
\item $M$ is semiprimitive and non-regular, and $T$ is transitive, or 
\item $M$ is regular and perfect and $T$ acts faithfully on each $T$-orbit.
\end{itemize}
\end{thm}

In the unrestricted case, we offer the following.

\begin{lem}
Suppose that $I$ is infinite and either $\res(T) \neq 1$ or $M$ is non-regular. Then $M\Wr T$ is not semiprimitive.
\end{lem}
\begin{proof}
Let $t\in \res(T)$, $h\in M^I$ and suppose $i\notin \mathrm{supp}(t)$. Then 
$$[h,t](i) = [h^{-1}h^t](i) = h^{-1}(i) h^t(i) =h^{-1}(i)  h(it^{-1}) = h^{-1}(i)h(i) = 1.$$
Since $\mathrm{supp}(t)$ is finite, we have $[\res(T),M^I] \leqslant \res(M^I)$. In particular, $\res(M^I)\res(T)$ is a normal subgroup of $M \Wr T$ that is non-regular (since either $\res(T) \neq 1$ or $\res(M^I)$ is non-regular) and intransitive on $\Delta^I$. Thus $M \Wr T$ is not semiprimitive.
\end{proof}

In particular, a group not handled by the above lemma is $\sym{2} \Wr \mathbb Z$.

\section{Properties of semiprimitive groups}

\label{sec:properties}

\subsection{Semiprimitive groups containing cycles of odd prime length}
A well-known result due to Jordan states that a primitive permutation group of degree $n$ containing a cycle of  prime length $p$ with  $p\leqslant n-3$ must contain $\mathrm{Alt}(n)$.  This result was extended to finite quasiprimitive  and finite innately transitive groups in \cite{praegershalev} and \cite{bamberg} respectively. Here we show that the same result holds in the context of semiprimitive groups (of arbitrary cardinality).

\begin{lem}
\label{lem:p cycles}
Let $G \leqslant \mathrm{Sym}(\Omega)$ be semiprimitive and suppose that $G$ contains a $p$-cycle for some  prime $p$. Then $G$ is  primitive.
\end{lem}
\begin{proof}
Suppose for a contradiction that   $\mathcal P$ is a non-trivial system of imprimitivity for  $G$ and let $M$ be the kernel of the action of $G$ on $\mathcal P$.  Since  $M$ is intransitive, $M$ is semiregular.  Let $g\in G$ be a $p$-cycle.  If $g\in M$, then since every element of $M$ is semiregular, we have $|\Omega|=p$, and so $G$ is primitive, a contradiction to $M$ being intransitive. Hence $g\notin M$. Let $\omega_1$, \dots, $\omega_p$ be the points of $\Omega$ moved by $g$. Suppose there is $\delta \in \mathcal P$ and $i \neq j$ such that   $\{\omega_i, \omega_j\} \subseteq \delta$. Then $\langle g\rangle $  fixes $\delta$  and hence $\{\omega_1,\dots,\omega_p\} \subset \delta$. Since $g$ fixes all other elements of $\Omega$, this gives $g\in M$, a contradiction.   Thus each $\omega_i$ lies in a distinct element of $\mathcal P$, $\delta_i$ say, and $g$ induces a $p$-cycle on $\{\delta_1,\dots,\delta_p\}$. If $|\delta_1| > 1$, then $g$ fixes some element in $\delta_1$, and so $g$ must fix $\delta_1$, a contradiction. Hence $|\delta_1|=1$ and so $\mathcal P$ is a trivial partition,  a final contradiction.
\end{proof}

\subsection{Bounds on orders of finite semiprimitive groups}

A classical result due to Bochert states that every primitive subgroup  $G$ of $\mathrm{Sym}(n)$ not containing $\mathrm{Alt}(n)$ satisfies
$$|\mathrm{Sym}(n):G| \geqslant \left \lfloor \frac{n+1}{2}\right \rfloor!.$$
This result was extended to the innately transitive setting in \cite[Theorem 6.1(3)]{bamberg}. Below we extend this result to the class of semiprimitive groups.

\begin{lem}
Let $G \leqslant \mathrm{Sym}(n)$ be a semiprimitive group such that $\mathrm{Alt}(n) \nleqslant G$. Then
$$|\mathrm{Sym}(n):G| \geqslant \left \lfloor \frac{n+1}{2}\right \rfloor!.$$
\end{lem}
\begin{proof}
It was noted in \cite{bamberg} that the proof given in \cite{bochert} depends merely upon the fact that a primitive permutation group of  degree $n$ containing a 3-cycle must contain $\mathrm{Alt}(n)$. By  Lemma~\ref{lem:p cycles}, a semiprimitive group of degree $n$ containing a 3-cycle must be primitive, and therefore contain $\mathrm{Alt}(n)$. Hence the proof given in \cite{bochert} applies here for semiprimitive groups.
\end{proof}

\section{Open problems}

\label{sec:problems}

\subsection{Orders of finite semiprimitive groups relative to degree}

The question of bounding the order of a primitive permutation group relative to its degree goes back to Jordan. Since then many results have  led to the notion that, aside from the symmetric and alternating groups of degree $n$, all primitive groups of degree $n$ are ``small''. A result of Praeger and Saxl \cite{praegersaxl} states that every primitive permutation group of degree $n$ that does not contain $\mathrm{Alt}(n)$ has order at most $4^n$. By omitting certain types of primitive groups, asymptotically better bounds can be given \cite{babai,pyber}. Further improvements can be obtained by employing The Classification of Finite Simple Groups \cite{cameron,liebeck}, with a sharp bound due to Mar\'{o}ti \cite{Maroti}. The analogous questions for quasiprimitive and innately transitive groups have been considered \cite{bamberg,praegershalev}, and the above bound of $4^n$ for the order of an innately transitive permutation group of degree $n$ not containing $\mathrm{Alt}(n)$ also holds. Thus we pose the following:

\vspace{.25cm}
\textbf{Problem 1:} What is the largest family $\mathcal F$ of semiprimitive groups     such that, for every semiprimitive group $G \in \mathcal F$ of degree $n$, we have $|G| \leqslant 4^n$?

\vspace{.25cm}
It is possible that the family described above will need to be defined by forbidding quotient actions, rather than just subgroups.

\subsection{Density of finite semiprimitive groups}

For a subset $N \subset \mathbb N$ and $x\in \mathbb N$ we define $N(x) = | \{ n\in N \mid n \leqslant x \}|$. The \emph{density} of the subset $N$ is  defined to be $\lim_{x\rightarrow \infty} N(x) / x$. We let 
$$ \mathrm{Deg}_{\mathrm{prim}} = \{ n \in \mathbb N \mid \text{ there is } G\leqslant \mathrm{Sym}(n) \text{ such that } G \text{ is primitive and } \mathrm{Alt}(n) \nleqslant G\}.$$
A result of Cameron, Neumann and Teague \cite{CNT} shows that the set $\mathrm{Deg}_{\mathrm{prim}}$ has density zero in the natural numbers. This leads us to the vague statement that ``for most degrees, the only primitive groups are either alternating or symmetric groups''.
Similarly we define
$$ \mathrm{Deg}_{\mathrm{it}} = \{ n \in \mathbb N \mid \text{ there is } G\leqslant \mathrm{Sym}(n) \text{ such that } G \text{ is innately transitive and } \mathrm{Alt}(n) \nleqslant G\}.$$
The analogous result, that $ \mathrm{Deg}_{\mathrm{it}}$ has density zero in the natural numbers, was established in \cite{HBPS}. We are lead to consider the density of the degrees of semiprimitive groups. Since every regular group is semiprimitive, we define the following:
\begin{eqnarray*} 
 \mathrm{Deg}_{\mathrm{sp}} =  \{ n \in \mathbb N \mid&  \text{ there is } G\leqslant \mathrm{Sym}(n) \text{ such that } G \text{ is semiprimitive, } \\
  & \text{ non-regular and }  \mathrm{Alt}(n) \nleqslant G \}.
\end{eqnarray*}
For each odd integer $n$, the dihedral group of order $2n$ is semiprimitive in its action on $n$ points. Thus the density of the set $\mathrm{Deg}_{\mathrm{sp}}$ is at least $\frac{1}{2}$. We are thus lead to the following:

\vspace{.25cm}
\textbf{Problem 2a:} What is the density of the set $\mathrm{Deg}_{\mathrm{sp}}$? Is it less than 1?

\vspace{.25cm}
In fact, it may be of greater interest to answer the following question:

\vspace{.25cm}
\textbf{Problem 2b:} What is the largest family $\mathcal F$ of finite semiprimitive groups such that the density of 
$\mathrm{Deg}_{\mathcal F} = \{n \in \mathbb N \mid \text{there is } G \leqslant \mathrm{Sym}(n) \text{ such that }  G\in \mathcal F\}$ 
 is zero?

\subsection{Base sizes of semiprimitive groups}

A \emph{base} of a permutation group $G \leqslant \mathrm{Sym}(\Omega)$ is a subset $B \subset \Omega$ such that $G_{(B)}=1$. Since each element of $G$ can be described uniquely by its action on a base, finding small (relative to degree) bases   is of interest. The \emph{base size} $\mathrm b_\Omega(G)$ (or $b(G)$) of $G$,   is the size of a smallest base (clearly $b_\Omega(G) \leqslant |\Omega|$). Some permutation groups have large bases, such as the alternating and symmetric groups (in their action on $n$ points, the base sizes of $\mathrm{Alt}(n)$ and  $\mathrm{Sym}(n)$ are $n-2$ and $n-1$, respectively). At the other end of the spectrum, a regular group has base size 1. Pyber conjectured \cite{pyberbasesize} that the base size of a primitive permutation group $G$ of degree $n$ is at most $\mathrm O\left(\frac{\log(|G|)}{\log(n)}\right)$.  The conjecture is known to be true for primitive groups without a regular elementary abelian normal subgroup; for more information, we refer the reader to  recent work of  Burness and Seress \cite{timseress}.

For semiprimitive groups, we pose the following:

\vspace{.25cm}
\textbf{Problem 3:} Investigate the base sizes of semiprimitive permutation groups.

\vspace{.25cm}
It was shown in \cite{praegershalev} that there exists a constant $n_0$ such that  a quasiprimitive group of degree $n\geqslant n_0$ not containing $\alt{n}$ has base size at most $4\sqrt{n}\log(n)$ (this generalises the result of Babai for uniprimitive groups \cite{babai}). We remark that we do not know of any semiprimitive group $G$   of degree $n$ that is not innately transitive  such that $b(G) > 4\sqrt{n}\log(n)$. In the infinite case, the study of base sizes of algebraic groups was initiated in recent work of Burness, Guralnick and Saxl \cite{burnessinfinitebase}.

\subsection{Minimal degrees of semiprimitive groups}

For a permutation group $G \leqslant \sym{\Omega}$ and $g\in G$, the \emph{support} of $g$ is $\mathrm{supp}(g)=\{ \omega \in \Omega \mid \omega^g \neq \omega \}$ and the \emph{degree} of $g$ is $\mathrm{deg}(g) = |\mathrm{supp}(g)|$. The minimal degree of $G$ is then the minimum of the degrees of the non-trivial elements of $G$.

Since  $\sym{n}$ contains transpositions and  $\alt{n}$ contains 3-cycles, we have $m(\sym{n})=2$ and $m(\alt{n})=3$. For other finite primitive permutation groups, the number is usually much higher (relative to the degree). In \cite{bamberg} it was shown that, if $G\leqslant \sym{n}$ is  innately transitive  with $\alt{n} \nleqslant G$, then $m_\Omega(G) \geqslant  (\sqrt{n}-1)/ 2$. Thus:

\vspace{.25cm}
\textbf{Problem 4a:} What is the largest family $\mathcal F$ of finite semiprimitive groups such that, for each $G\in \mathcal{F}$ of degree $n$, $m(G) \geqslant (\sqrt{n}-1)/2$?
\vspace{.25cm}

If $\Omega$ is an infinite set and $G \leqslant \sym{\Omega}$ is primitive, then the so-called Jordan-Wielandt theorem shows that if $m(G)$ is finite then $G$ contains the finitary alternating group on $\Omega$, that is, the group of permutations of $\Omega$ with finite support and even degree. We do not know of any extension of the Jordan-Wielandt theorem to infinite quasiprimitive groups, and we thus pose the following:

\vspace{.25cm}
\textbf{Problem 4b:} If $\Omega$ is an infinite set and $G \leqslant \sym{\Omega}$ is semiprimitive, does $G$ contain the finitary alternating group?

\subsection{Normalisers of semiprimitive groups}

Suppose that $G$ is a primitive permutation group of degree $n$. A recent result of Guralnick, Mar\'{o}ti and Pyber states that, apart from finitely many explicitly described exceptions,  $|\mathrm N_{\mathrm{Sym}(n)}(G) : G | < n$ -- see \cite{GMP}. Concerning normalisers, we first record the following:
\begin{lem}
Suppose that $G$ is a transitive subgroup of $\mathrm{Sym}(n)$ and let $N = \mathrm N_{\mathrm{Sym}(n)}(G)$. 
\begin{enumerate}[(i)]
\item If $G$ is primitive, then $N$ is primitive.
\item If $G$ is innately transitive, then $N$ is innately transitive.
\end{enumerate}
\end{lem}
\begin{proof}
Part (i) is clear: a block of imprimitivity for $N$ is a block of imprimitivity for $G$.

For part (ii), note that for any $x\in N$ and plinth $K$ of $G$, $K^x$ is also a plinth of $G$. Thus if $K \neq K^x$ for some $x\in N$, then $G$  is an innately transitive group with at least two plinths, and so is primitive of type HS or HC by Lemma~\ref{lem:prim of type hs hc}. Thus $N$ is primitive by part (i). We may thus assume that $K=K^x$ for all $x\in N$, and so $K$ is normal in $N$. Hence $K$ is a transitive minimal normal subgroup of $N$, so  $N$ is innately transitive.\end{proof}

%
%
Thus normalisers of primitive, quasiprimitive and innately transitive groups are tightly controlled in terms of their actions. We therefore pose the following:
\vspace{.25cm}

\textbf{Problem 5a:} Bound $|\mathrm N_{\mathrm{Sym}(n)}(G) : G|$ if $G$ is quasiprimitive.

\textbf{Problem 5b:} Bound $|\mathrm N_{\mathrm{Sym}(n)}(G) : G|$ if $G$ is innately transitive.

\textbf{Problem 5c:} Bound $|\mathrm N_{\mathrm{Sym}(n)}(G) : G|$ if $G$ is semiprimitive.
\vspace{.25cm}

For the third problem above,  a linear bound in $n$ is not possible, consider regular elementary abelian groups for example. In fact non-regular examples exist, for an odd prime $p$, and an integer $d$, take $(\mathrm C_p)^d \rtimes \mathrm C_2$, where an involution  acts on $(\mathrm C_p)^d$ by inversion. This group is semiprimitive of degree $p^d$, and the normaliser is $(\mathrm C_p)^d \rtimes \mathrm{GL}(d,p)$. We are thus lead to consider whether a bound of the form $n^{c\log n}$ for some constant $c$ would suffice in Problem 6c. In fact, \cite[Theorem 1.7]{GMP} shows that a bound of the form $4^{\frac{n}{\sqrt{\log n}}} n^{\log n}$ holds simply under the assumption that $G$ is transitive, so the problem is to decide if the first term may be dropped for semiprimitive groups $G$.

Note that the normaliser of a quasiprimitive group may not be quasiprimitive. Indeed, the normaliser of the action on 12 points of $\mathrm{Alt}(5)$ is an innately transitive group of type ASQ$^{\text{non-reg}}$. Finally, we remark that the normaliser of a semiprimitive group need not be semiprimitive. For example, take $G=\mathrm D_8$ in its  regular representation of degree $8$. Then $G$ is semiprimitive, but $\mathrm N_{\mathrm{Sym}(8)}(G)$ is a non-regular 2-group and so fails to be semiprimitive.

\subsection{Graph theory}

Recalling the definition of graph-restrictive permutation groups  from Section~\ref{sec:graphs}, we mention again the problem that motivated this work. The conjecture below may be the most intractable problem discussed in this section. 

\begin{conj*}[Poto\v{c}nik-Spiga-Verret \cite{psv}]
A permutation group is semiprimitive  if and only if it is graph-restrictive.
\end{conj*}


\begin{thebibliography}{99}
\bibitem{babai} Babai, L\'{a}szl\'{o}.  On the order of uniprimitive permutation groups. Ann. of Math. (2) 113 (1981), 553--568.

\bibitem{bambergetal}
Bamberg, John, Giudici, Michael, Morris, Joy, Royle, Gordon F. and  Spiga, Pablo, Generalised quadrangles with a group of automorphisms acting primitively on points and lines. J. Combin. Theory Ser. A 119 (2012), no. 7, 1479--1499. 

\bibitem{bambergpraeger} Bamberg, John and Praeger, Cheryl E.~ Finite permutation groups with a transitive minimal normal subgroup. Proc. London Math. Soc. 89(3)  (2004) 71--103.

\bibitem{bamberg}
Bamberg, John.
Bounds and quotient actions of innately transitive groups.  
J. Aust. Math. Soc. 79 (2005), no. 1, 95--112. 

\bibitem{bermar} Bereczky, {\'A}ron and Mar{\'o}ti, Attila.  On groups with every normal subgroup transitive or semiregular. J. Algebra. 319(4) (2008), 1733--1751.

\bibitem{bochert} Bochert, A. 
Ueber die Zahl der verschiedenen Werthe, die eine Function gegebener Buchstaben durch Vertauschung derselben erlangen kann. Math. Ann. 65 (1889), 584--590.

\bibitem{burnessinfinitebase}
Burness, Timothy C., Guralnick, Robert M.~ and Saxl, Jan.
On base sizes for algebraic groups.
Journal of the European Mathematical Society, to appear.

\bibitem{timseress}
Burness, Timothy C.~ and Seress, \'Akos.
On Pyber's base size conjecture.  
Trans. Amer. Math. Soc. 367 (2015), no.8, 5633--5651. 



\bibitem{cameron}
Cameron, Peter J.~
Finite permutation groups and finite simple groups. 
Bull. London Math. Soc. 13 (1981), no. 1, 1--22. 

\bibitem{CNT} Cameron, Peter J., Neumann, Peter M.~ and Teague, David N.~ On the degrees of primitive permutation groups. Math. Z. 180 (1982), no.2, 141--149.

\bibitem{DM} Dixon, John D., Mortimer, Brian. ``Permutation groups'', Graduate Texts in Mathematics, 163. Springer-Verlag, New York, 1996.

\bibitem{lukemichael3n2} Giudici, Michael and Morgan, Luke. A class of semiprimitive groups that are graph-restrictive. Bull. London Math. Soc. 46 (2014) 1226--1236.

\bibitem{GMP} Guralnick, Robert M., Mar\'{o}ti, Attila and Pyber, L\'{a}szl\'{o}. Normalisers of Primitive Permutation groups. arXiv preprint: http://arxiv.org/abs/1603.00187v1.

\bibitem{HBPS} Heath-Brown, D. R., Praeger, Cheryl E. and Shalev, Aner. Permutation groups, simple groups, and sieve methods. Probability in mathematics. Israel J. Math. 148 (2005), 347--375. 

\bibitem{Isaacs} Isaacs, I. Martin, ``Finite group theory", American Mathematical Society, Graduate Studies in Mathematics, 92, 2008.




\bibitem{liebeck}
Liebeck, Martin W.~
On minimal degrees and base sizes of primitive permutation groups. 
Arch. Math. (Basel) 43 (1984), no. 1, 11--15. 



\bibitem{LPSON}
Liebeck, Martin W., Praeger, Cheryl E.~ and Saxl, Jan.
On the O'Nan--Scott theorem for finite primitive permutation groups.
J. Austral. Math. Soc. (Series A) 44 (1988), 389--396.

\bibitem{macpraeger}
Macpherson, Dugald and Praeger, Cheryl E. Infinitary versions of the O'Nan-Scott theorem. Proc. London Math. Soc. (3) 68 (1994), no. 3, 518--540.


\bibitem{Maroti} Mar\'{o}ti, Attila.
On the orders of primitive groups. 
J. Algebra 258 (2002), no. 2, 631--640.

\bibitem{McDonough}
McDonough,  T.~P., A permutation representation of a free group, Quart. J. Math. Oxford (2), 28 (1977), 353--356.

\bibitem{msv} Morgan, Luke and Spiga, Pablo and Verret, Gabriel.
On the order of Borel subgroups of group amalgams and an application to locally-transitive graphs. 
J. Algebra 434 (2015), 138--152. 

\bibitem{psv} Poto\v{c}nik, Primoz, Spiga, Pablo and Verret, Gabriel. On graph-restrictive permutation
groups, J. Combin. Theory Ser. B 102 (2012), 820--831.

\bibitem{praegersaxl} Praeger, Cheryl E.~ and Saxl, Jan. On the orders of primitive permutation groups. Bull. London. Math. Soc. 12, no.4, (1980), 303--307.

\bibitem{cheryl-imp sym}
Praeger, Cheryl E. Imprimitive symmetric graphs. 
Ars Combin. 19 (1985), A, 149--163. 


\bibitem{praeger2-trans}Praeger, Cheryl E. An {O}'{N}an-{S}cott theorem for finite quasiprimitive
              permutation groups and an application to {$2$}-arc transitive
              graphs. J. London Math. Soc. (2), 47, (1993), 227--239.

\bibitem{PLNON}
Praeger, Cheryl E., Li, Cai-Heng and Niemeyer, Alice C.
Finite transitive permutation groups and finite vertex-transitive graphs. Graph symmetry (Montreal, PQ, 1996), 277--318, 
NATO Adv. Sci. Inst. Ser. C Math. Phys. Sci., 497, Kluwer Acad. Publ., Dordrecht, 1997. 



\bibitem{PraegerON}
Praeger, Cheryl E.~
 Finite quasiprimitive graphs, Surveys in combinatorics, 1997. Proceedings of the 16th British combinatorial conference, London, UK, July 1997 (R. A. Bailey, ed.), Lond. Math. Soc. Lect. Note Ser., no. 241, Cambridge University Press, 1997, pp. 65--85. 
 
 \bibitem{praegerc}
Praeger, Cheryl E.~ Finite quasiprimitive group actions on graphs and designs, in: Young Gheel Baik, David L. Johnson, Ann Chi
Kim (Eds.), Groups -- Korea, de Gruyter, Berlin, New York, (2000), pp. 319--331.

\bibitem{praegerdt}
Praeger, Cheryl E., Saxl, Jan, and  Yokoyama, Kazuhiro,
Distance transitive graphs and finite simple groups. 
Proc. London Math. Soc. (3) 55 (1987), no. 1, 1--21. 

\bibitem{praegershalev}
Praeger, Cheryl E. and Shalev, Aner.
Bounds on finite quasiprimitive permutation groups. 
Special issue on group theory. 
J. Aust. Math. Soc. 71 (2001), no. 2, 243--258. 

\bibitem{pyber}
Pyber, L\'{a}szl\'{o}.
On the orders of doubly transitive permutation groups, elementary estimates.
J. Combin. Theory Ser. A 62 (1993), no. 2, 361--366. 

\bibitem{pyberbasesize} Pyber, L\'{a}szl\'{o}.
Asymptotic results for permutation groups.  Groups and computation (New Brunswick, NJ, 1991), 197--219.


\bibitem{csabahendrik} Schneider, Csaba and Van Maldeghem, Hendrik Primitive flag-transitive generalized hexagons and octagons. J. Combin. Theory Ser. A 115 (2008), no. 8, 1436--1455.

\bibitem{simon} Smith, Simon S.~ 
A classification of primitive permutation groups with finite stabilizers. 
J. Algebra 432 (2015), 12--21. 

\bibitem{pablo} Spiga, Pablo. Two local conditions on the vertex stabiliser of arc-transitive graphs and their effect on the Sylow subgroups. J. Group Theory 15 no. 1 (2012), 23--35. 

\bibitem{trofweiss1}
Trofimov, V.~ I. and Weiss, R.~ M.  Graphs with a locally linear group of automorphisms. Math. Proc. Cambridge Philos. Soc. 118 (1995), no. 2, 191--206. 

\bibitem{verret}
Verret, G., On the order of arc-stabilisers in arc-transitive graphs. Bull. Aust. Math. Soc. 80 (2009), 498--505.


\bibitem{weissc}
Weiss, R. s-transitive graphs. Algebraic methods in graph theory, Vol. I, II (Szeged, 1978), pp. 827--847, Colloq. Math. Soc. J\'{a}nos Bolyai, 25, North-Holland, Amsterdam-New York, 1981. 
\end{thebibliography}
\end{document}